\documentclass{article}
\usepackage[utf8]{inputenc}

\usepackage[a4paper,margin=.879in,footskip=0.25in]{geometry}

\usepackage{amsmath,amsthm} 
\usepackage{amssymb}
\usepackage{hyperref}
\usepackage[all]{xypic} 
\usepackage{amsbsy}
\usepackage{graphicx}  
\usepackage{subfigure} 
\usepackage[all,cmtip]{xy} 
\usepackage{multirow}  
\usepackage{multicol}
\usepackage{rotating} 
\usepackage{comment} 
\usepackage{tikz} 
\usepackage{tikz-cd}
\usetikzlibrary{arrows, arrows.meta, decorations.markings, decorations.pathmorphing, shapes} 
\usepackage{hyperref} 
\usepackage{lscape} 
\usepackage{manfnt} 
\usepackage{fancyvrb}  
\usepackage{todonotes} 
\usepackage{bbold}  

\numberwithin{equation}{section}

\newtheorem{thm}{Theorem}[section]
\newtheorem{cor}[thm]{Corollary}

\newtheorem{lem}[thm]{Lemma}
\newtheorem{prop}[thm]{Proposition}
\theoremstyle{definition} 
\newtheorem{defn}[thm]{Definition}
\newtheorem{rem}[thm]{Remark}
\newtheorem{exam}[thm]{Example}

\newcommand{\bC}{\mathbb{C}}
\newcommand{\bK}{\mathbb{K}}

\newcommand{\bH}{\mathbb{H}}
\newcommand{\bP}{\mathbb{P}} 
\newcommand{\bQ}{\mathbb{Q}}
\newcommand{\bR}{\mathbb{R}}
\newcommand{\bZ}{\mathbb{Z}}
\newcommand{\bN}{\mathbb{N}} 
 
\newcommand{\bT}{\mathbb{T}}

\newcommand{\cF}{\mathcal{F}}

\newcommand{\ccL}{\mathcal{L}}
\newcommand{\cM}{\mathcal{M}} 
\newcommand{\cO}{\mathcal{O}} 
\newcommand{\cP}{\mathcal{P}}
\newcommand{\ccR}{\mathcal{R}} 
\newcommand{\cS}{\mathcal{S}} 
\newcommand{\cT}{\mathcal{T}}

\newcommand{\fS}{\mathfrak{S}}

\newcommand{\Hom}{\operatorname{Hom}}
\newcommand{\Ext}{\operatorname{Ext}}

\newcommand{\End}{\operatorname{End}}

\newcommand{\Sym}{\operatorname{Sym}}

\newcommand{\GL}{\operatorname{GL}}
\newcommand{\PGL}{\operatorname{PGL}}
\newcommand{\fsl}{\mathfrak{s} \mathfrak{l}}

\newcommand{\Vect}{\text{\bf Vect}}

\newcommand{\Hilb}{\operatorname{Hilb}}

\newcommand{\Pic}{\operatorname{Pic}}
\newcommand{\Cl}{\operatorname{Cl}}

\newcommand{\Mov}{\operatorname{Mov}}
\newcommand{\Proj}{\operatorname{Proj}}
\newcommand{\sm}{\operatorname{sm}}

\newcommand{\sing}{\operatorname{sing}}
\newcommand{\Sp}{\operatorname{Sp}}
\newcommand{\SL}{\operatorname{SL}}
\newcommand{\SO}{\operatorname{SO}}

\newcommand{\gen}{\operatorname{gen}}
\newcommand{\Ex}{\textbf{Ex}}

\newcommand{\Id}{\operatorname{Id}}
\newcommand{\tr}{\text{tr}}

\title{Crepant resolutions of  stratified varieties via gluing}
\author{Daniel Kaplan\footnote{Universiteit Hasselt, daniel.kaplan@uhasselt.be} \  and Travis Schedler\footnote{Imperial College London, t.schedler@imperial.ac.uk}}
\date{February 2024}

\begin{document}

\maketitle
\tableofcontents

\abstract{
Let $X$ be a variety with a stratification $\cS$ into smooth locally closed subvarieties such that $X$ is locally a product along each stratum (e.g., a symplectic singularity). We prove that  assigning to each open subset $U \subset X$ the set of isomorphism classes of locally projective crepant resolutions of $U$ defines an $\cS$-constructible sheaf of sets. 
Thus, for each stratum $S$ and basepoint $s \in S$,  the fundamental group acts on the set of germs of projective crepant resolutions at $s$, leaving invariant the germs extending 
to the entire stratum.   Global locally projective crepant resolutions correspond to compatible such choices for all strata. 
For example, if the local projective crepant resolutions are unique, they automatically glue uniquely.

We give criteria for a locally projective crepant resolution $\rho: \tilde X \to X$ to be globally projective. We show that the sheafification of the presheaf $U \mapsto \Pic(\rho^{-1}(U)/U)$ of relative Picard classes is also constructible. The
resolution is globally projective only if there exist local relatively ample bundles whose classes glue to a global section of this sheaf.  The obstruction to lifting this section to a global ample line bundle is encoded by a gerbe on the singularity $X$. 
We show the gerbes are automatically trivial if $X$ is a symplectic quotient singularity. 

Our main results hold in the more general setting of partial crepant resolutions, that need not have smooth source.

We apply the theory to symmetric powers and Hilbert schemes of surfaces with du Val singularities, finite quotients of tori, multiplicative and Nakajima quiver varieties, as well as to canonical threefold singularities.  
}


\section{Introduction}

Let $X$ be a variety.
Consider the classification of all isomorphism classes of pairs $(Y, \rho: Y \rightarrow X)$ where $Y$ is 
a variety and $\rho$ is a birational morphism. The category of such pairs $(Y, \rho)$ is discrete:  given $(Y_1, \rho_1)$ and $(Y_2, \rho_2)$, if there is an isomorphism $\varphi: Y_1 \rightarrow Y_2$ satisfying $\rho_1 = \rho_2 \circ \varphi$ then $\varphi$ is unique. Consequently, the set of isomorphism classes forms a sheaf (whereas, a priori, one might only expect the categories of pairs to form a stack). 

In more detail, suppose we have an open covering $\bigsqcup U_i \rightarrow X$ of $X$ and birational morphisms $V_i \to U_i$.  We seek a variety $Y$ with a map to $X$ and  an open covering $\bigsqcup V_i \to Y$ such that the diagram
\[
\xymatrix{ V_i \ar[d] \ar[r] & Y \ar[d] \\ U_i \ar[r] & X
}
\]
commutes for all $i$. If such a $Y$ exists, it is unique up to unique isomorphism. Indeed, for each $i,j$, if $V_i \to U_i$ and $V_j \to U_j$ are isomorphic over a nonempty overlap $U_i \cap U_j$, then they uniquely glue, by the preceding paragraph. So the only condition for the existence of $Y$ is that the maps are pairwise compatible, i.e., every pair $V_i \to U_i$ and $V_j \to U_j$ restrict to isomorphic maps over the intersection $U_i \cap U_j$.

We will be especially interested in the case where each $V_i \to U_i$ is (locally) projective and crepant.  In this case, $Y \to X$ is locally projective and crepant.  As a special case, suppose that each $U_i$ admits a unique locally projective crepant resolution $V_i \to U_i$.   Then the gluing condition is automatic, and we conclude that $X$ admits a unique locally projective crepant resolution $Y \to X$.  The most well-known example of this phenomenon is if $X$ has only du Val surface singularities, in which case the resolution $Y \to X$ is projective and minimal. We give more interesting examples below.\\
\\
In this paper, we classify locally projective (partial) crepant resolutions from local data, under certain hypotheses. Our primary setting of interest is that of symplectic resolutions.  In order for $X$ to admit a symplectic resolution, it must be a symplectic singularity in the sense of Beauville \cite{Beauville}. According to Kaledin \cite{Kaledin06}, algebraic symplectic singularities have a 
finite  stratification by symplectic leaves.

Our main theorem shows that isomorphism classes of locally projective symplectic resolutions of $X$ are in bijection with compatible, monodromy-free choices of local symplectic resolutions at  basepoints of the strata. 
We also give criteria for when these resolutions are globally projective.

We actually prove a much more general statement for (partial) crepant resolutions, that does not require the symplectic structure, at the price of assuming the existence of a nice stratification. \\
  
 We provide two examples of quotient singularities 
where the utility of this approach is visible.


\begin{exam}
    Let $\fS_n$ act on $(\bC^\times)^{2n}$ by permuting the $n$ copies of $(\bC^\times)^2$.  Take the quotient $(\bC^\times)^{2n}/\fS_n \cong \Sym^n((\bC^\times)^2)$. Then the singular locus is the image of 
 the diagonal (where two pairs are equal).  At a singular point, there is a local neighborhood isomorphic to a product of singularities $\bC^{2m}/\fS_m$. These are well-known to admit a unique projective crepant resolution, by the Hilbert--Chow morphism (and this is also a special case of \cite[Proposition 1.2]{Bellamy16}, which we use below). So there is a unique locally projective crepant resolution, gluing these together.  This resolution is in fact globally projective, given by the Hilbert--Chow morphism $\Hilb^n((\bC^\times)^2) \to \Sym^n((\bC^\times)^2)$.  
\end{exam}

\begin{exam}
Fix $n \in \bN$. The group $B_n := C_2^n \rtimes \fS_n$ acts on $(\bC^\times)^{2n}$ where $\fS_n$ denotes the symmetric group permuting the $n$ copies of $(\bC^\times)^2$ and $C_2^n$ acts diagonally with $C_2$ acting on $(\bC^\times)^2$ via $(z,w) \mapsto (z^{-1},w^{-1})$.
The singular locus of the quotient $X := (\bC^\times)^{2n} / B_n$ is the union of the diagonal 
with the locus of $n$-tuples of pairs of complex numbers containing a pair of the form  $( \pm 1, \pm 1)$. Along a point of a diagonal-type stratum---that is, a stratum where some pairs are equal, but none are $(\pm 1, \pm 1)$---the singularity is a product of singularities $\bC^{2m}/ \fS_m$, which each admit a unique crepant resolution given by $\Hilb^{2m}(\bC^2)$.
Singularities where some number of pairs are one of the four pairs $(\pm 1, \pm 1)$ are the same as in a product of singularities of the form $\Sym^m(\bC^2/C_2)$, with $C_2$ now acting by $\pm I$.  Our results then imply the formula:
\begin{align*}
\# \text{\{isomorphism} & \text{ classes of } \text{locally projective crepant resolutions of } X \} \\
&= (\# \text{\{isomorphism classes of projective crepant resolutions of } \Sym^n(\bC^2 / C_2 ) \})^4 = n^4,
\end{align*}
where the last equality is a special case of Bellamy's formula\footnote{Bellamy is counting symplectic resolutions but the notions of symplectic and crepant agree whenever a symplectic resolution exists, 
see e.g. \cite[Proposition 3.2]{Kaledin03}.} in \cite[Proposition 1.2]{Bellamy16}. 
To construct the resolutions, independently pick a local projective crepant resolution around each of the most singular points (all pairs are equal and are one of the four pairs $(\pm 1, \pm 1)$), and these spread out uniquely to a global resolution. In fact, we can see that all such resolutions are \emph{globally} projective. See Section \ref{ss: symmetric_powers} for more details.
\end{exam}

Our key condition on the singularities is the following. Let $X$ be a complex analytic variety equipped with a locally finite stratification $\mathcal{S} = \{S_i\}$ by Zariski locally closed, connected, smooth subvarieties. (Note that, if $X$ is in fact a complex \emph{algebraic} variety, by which we mean a separated, reduced scheme of finite type over $\bC$, then the stratification must be finite.)

\begin{defn}\label{d:local product} 
$X$ is said to be \emph{locally a product along $S_i$} if, for every $s \in S_i$, there is a connected neighborhood $U \ni s$ in $X$, a pointed stratified variety $F \ni 0$, and a stratified isomorphism $\varphi: (U \cap S_i) \times F \to U$  which restricts to the inclusion on $(U \cap S_i) \times \{0\}$. We call the triple $(U, F, \varphi)$ a \emph{local product neighborhood} and refer to $F$ as a \emph{local slice} to $S_i$.
\end{defn}
Note that the condition implies that the pointed analytic germs of $X$ along $s \in S_i$ are locally constant in $s$. In particular, since $S_i$ is connected, all germs of $X$ along points of $S_i$ are isomorphic. 

Our main results concern varieties $X$ with a locally finite stratification $\cS$ by such strata, such that $F$ is stratified and the isomorphism $\varphi$ respects the stratification. We are interested in crepant resolutions of $X$, so we assume that $X$ is normal and that its canonical divisor is $\bQ$-Cartier. For technical purposes we will also need to assume that $F$ admits a projective 
crepant resolution $\tilde F \to F$  with $H^1(\tilde F, \cO)=0$ (see Definition \ref{d: prmm}), in other words local resolutions can be obtained by resolving the slice.  This holds in all our examples; for the general case, see Remark \ref{r:bchm}.

Since $X$ is assumed to be normal, and our neighborhoods $U$ are connected, they are irreducible. 
\begin{defn}
For an open set $U \subset X$, let $\ccR^s(U)$ denote the set of isomorphism classes of (smooth) locally projective, crepant resolutions of $U$. 
\end{defn}

Note that an isomorphism of \emph{resolutions} is stronger than that of abstract varieties as it is required to induce the identity map over the base. 

In the main body of the text, we will relax the condition that the source is smooth and study partial crepant resolutions locally dominated by a terminal crepant resolution. 

The assignment $U \mapsto \ccR^s(U)$ 
is a presheaf of sets. In fact, is is a sheaf, thanks to the uniqueness of gluing. Our main result is then:
\begin{thm}\label{t:constr}
For a variety $X$ with stratification as above, the sheaf $\ccR^s$ 
is $\cS$-constructible. 
\end{thm}



\begin{exam}
Suppose that $X$ has two strata: the open stratum $S_0$ (the smooth locus), and one singular connected stratum  $S_1$, along which $X$ is locally a product. Pick a basepoint $s_1 \in S_1$.
The theorem then implies that
the fundamental group $\pi_1(S_1, s_1)$ acts on the stalk $\ccR^s_{s_1}$ of the sheaf $\ccR^s$ at $s_1$.
The fixed points of this action are the resolutions that extend to a (locally projective, crepant) resolution on a neighborhood $U_{S_1}$ of all of $S_1$.
Since these resolutions are isomorphisms over $S_0 \cap U_{S_1}$, they glue to the identity resolution $S_0 \to S_0$, to give a global locally projective crepant resolution of $X$.
So $\ccR^s(X)$ can be viewed as the subset of fundamental group invariant resolutions of $\ccR^s_{s_1}$.  
\end{exam}

\begin{exam}
Let $X$ be singular with stratification $\cS := \{ S_0, S_1, S_2 \}$, such that $\overline{S_i} \supset S_{i+1}$ for $i=0, 1$.
Again we have a minimum stratum, $S_2$, and so can pick a crepant resolution in a neighborhood of some $s_2 \in S_2$ and then ask if it extends to neighborhoods $U_{S_2}$, $U_{S_1}$, and then $U_{S_0}$. In this way $\ccR^s(X)$ is again the subset of $\ccR^s_{s_2}$ invariant under the fundamental groups of $S_2$ and $S_1$. In words, crepant resolutions are determined by local crepant resolutions around their ``most singular" points.  
\end{exam}

The sheaf $\ccR^s$ is governed by
a fundamental object:
the presheaf $\cP$ of relative rational Picard groups, $U \mapsto \Pic(\rho^{-1}(U)/U) \otimes_{\bZ} \bQ$ (see Theorem \ref{t:pic} below).   
Restricted to the codimension two strata (where there can only be du Val singularities), its sheafification  recovers the local systems defined by Namikawa to understand  the deformation theory of symplectic singularities in \cite{Namikawa11}; these local systems determine the Namikawa Weyl (or symplectic Galois) group \cite{Namikawa10}, with monodromy indicating the presence of non-simply laced types. This group was recently shown to act on cohomology of fibers of symplectic resolutions, defining a vast generalization of Springer theory \cite{McN-Springer}.

The theorem implies the following description of the possible projective crepant resolutions:

\begin{cor} \label{c:main}
Fix $(X$, $\cS)$ a stratified variety as above. Pick a basepoint $s_i \in S_i$ for each stratum $S_i \in \cS$. Then isomorphism classes of resolutions in  $\ccR^s(X)$ 
are in bijection, via restriction, with
the set of locally projective, crepant, compatible, $\pi(S_i, s_i)$-invariant local resolutions of neighborhoods of the $s_i$.
\end{cor}

The second goal of this paper is to express the aforementioned compatibility conditions in terms of linear algebra, thereby establishing a framework to build global resolutions from local ones.

 It is well known that a local system of sets on a space is equivalent to a functor from the fundamental groupoid to sets; by choosing a basepoint this is then equivalent to an action of the fundamental group on the fiber at the basepoint. In the stratified  setting, MacPherson explained that constructible sheaves can similarly be defined as functors from the exit path category (see \cite[Theorem 1.2]{TrEP}). Let us recall the statement.
 
 \begin{defn}
 Given a locally finite stratified topological space $X = \sqcup_i S_i$ with $S_i$ locally closed smooth manifolds 
 an \emph{exit path} is a path $\gamma: [0,1] \to X$ such that, for $0 \leq t_1 < t_2 \leq 1$, the dimension of the stratum containing $\gamma(t_1)$  is less than or equal to the dimension of the stratum containing $\gamma(t_2)$. 
 \end{defn}
 Exit paths are all concatenations of paths of a simpler form: 
 
 \begin{defn}
 A \emph{simplified exit path} is $\gamma: [0,1] \to X$ such that $\gamma((0,1])$ all lies in the same stratum. 
 \end{defn}
 
 So, a simplified exit path either lies in a single stratum, or immediately exits one stratum to another.  All exit paths are finite concatenations of simplified exit paths (up to parameterization).  
 
 \begin{defn} Let $\Ex(X,\cS)$ be the category whose objects are points $x \in X$ and whose morphisms are tame homotopy classes of exit paths, through exit paths.  
 \end{defn}
 
 Here,  following \cite{TrEP}, a tame homotopy is a map $H: [0,1] \times [0,1] \to X$ such that the source can be continuously triangulated with faces mapping to the same stratum. We will not need this notion in this paper.

\begin{thm}[MacPherson; see \cite{TrEP}]
There is an equivalence between $\cS$-constructible sheaves and functors $F: \bf{Ex}(X,\cS) \to \bf{Sets}$. This equivalence sends a functor $F$ to the sheaf $\cF$ given by:
 \[
 \cF(U) = \Gamma(U,F) := \left \{ (\rho_x)_{x \in U} \in \prod_{x \in U} F(x) \ \middle | \  \rho_y = F([\gamma])(\rho_x),  \text{ for all exit paths $\gamma: x \to y$ in $U$} \right \}.
 \]
 \end{thm}
 
Thus, Theorem \ref{t:constr} can be rephrased as the following:

 \begin{cor}\label{c:exitp}
 There is a functor $R: 
 {\bf Ex}(X,\cS) \to \bf{Sets}$, sending $x \in X$ to the stalk $\ccR^s_x$,
 such that $\ccR^s(U)=\Gamma(U,R)$.
 The same is true for the functor giving only the smooth resolutions.
 \end{cor}
 
 To keep the prerequisites to a minimum, we will not rely on  MacPherson's equivalence.
 We will instead recall and use  the following easier part of MacPherson's result:
 There is a unique parallel transport on exit paths, such that global sections of $\cF$ are the same as exit path compatible choices of local sections.  A proof is provided for the reader's convenience in Section \ref{ss: exit paths}.
 
 If we label by $S_i$ the strata of $X$ and choose basepoints $s_i \in S_i$, it is clear that $f \in \Gamma(X,R)$ is uniquely determined by its values in the stalks $R(s_i)$: we take the ones compatible under exit paths with endpoints among the $s_i$. We deduce a more precise version of Corollary \ref{c:main}, where the compatibility between local resolutions is given by one being sent to another (up to isomorphism) via exit paths with endpoints the $s_i$.
  Note that everything is determined by the local resolutions on the closed strata (although to express compatibility we need to consider exit paths to non-closed strata). 


To express the isomorphism classes of local projective crepant resolutions, we use that two such local resolutions are related by a birational transformation given by a line bundle.  
Fix $\tilde{X}$ a projective crepant resolution of $X$ (if one exists). 
Then isomorphism classes of projective crepant resolutions correspond to Mori chambers in the movable cone $\text{Mov}(\tilde{X}/X)\subseteq \Pic(\tilde X/X)\otimes_{\bZ} \bQ$, associating each resolution to its ample
cone. (Note that, when $X$ is affine algebraic, it is a consequence of \cite[Corollary 1.3.2]{BCHM10} that there are only finitely many such chambers; see Remark \ref{r:bchm} below for more details.)

Applying this to strata, we can express parallel transport along exit paths in terms of linear monodromy on the relative Picard group. Given a locally projective crepant resolution $\rho: \tilde X \to X$, let $\cP^{\tilde X}_\bZ$ be the presheaf of relative Picard groups, $\cP^{\tilde X}_\bZ(U) := \Pic(\rho^{-1}(U)/U)$, and let $\cP := \cP^{\tilde X}_{\bZ} \otimes_{\bZ} \bQ$ be the associated presheaf of rational vector spaces. 
\begin{rem}
    Since any two crepant resolutions coincide outside codimension two subsets, these presheaves do not depend on the choice of $\tilde X$: any other locally projective smooth crepant resolution $\tilde X' \to X$ will canonically have the same class group of Weil divisors up to linear equivalence, hence the same Picard group, and then the same relative Picard group since pullbacks of line bundles on $X$ agree.  So we can omit the $\tilde X$. However, for  consistency with notation in the main body of the article, we keep the notation in the integral case (it becomes important when we allow $\tilde X$ to be singular).
    \end{rem}

\begin{thm}\label{t:pic}
Assume that $X$ is equipped with a stratification satisfying our hypotheses, and that $\rho: \tilde X \to X$ is a locally projective crepant resolution.
Then the presheaf $\cP$ of local relative rational Picard groups
has
sheafification $\cP^\sharp$ which is a constructible sheaf of rational vector spaces. 
The sheaf $\ccR^s$ is the sheafification of the associated presheaf of sets of Mori chambers.
\end{thm}
As a result of this, we can give a necessary condition for a locally projective crepant resolution to be globally projective: the section of Mori chambers should be represented by  a global section of  $\cP^\sharp$, i.e., consistent choice of 
local relatively ample rational class.
In many cases we will see that this is almost a biconditional; in particular we will show:
\begin{thm}\label{t:qfact-symp-intro} Let $\tilde X \to X$ be a locally projective crepant resolution of a $\bQ$-factorial symplectic singularity (e.g., a symplectic quotient singularity) with a finite symplectic stratification. Then $\tilde X$ is globally projective if and only if there exists an associated section $\sigma \in \Gamma(X,\cP^\sharp)$ of relatively ample
 classes. Such a section uniquely determines a relatively ample class in $\Pic(\tilde X/X)\otimes_\bZ \bQ$.
\end{thm}
\begin{rem}
For general $X$, it is not true that the existence of a consistent section of  
local relatively ample classes is enough to have a relatively ample bundle on all of $\tilde X$, i.e., for $\tilde X \to X$ to be globally projective. See Example \ref{rem:example_globally_proj}.
\end{rem}
For general $X$, we will prove:
\begin{thm}\label{t:gerbe-intro}
    Given a locally projective crepant resolution $\tilde X \to X$ and a section $\tilde \sigma \in \Gamma(X,\cP_\bZ^{\tilde X,\sharp})$ of relatively ample classes for $\tilde X \to X$, 
    there is a canonical gerbe $\ccL_{\tilde \sigma}$ on $X$. Then $\tilde X \to X$ is globally projective with  relatively ample line bundle locally in the class $\tilde \sigma$ if and only if $\ccL_{\tilde \sigma}$ is trivial.
\end{thm}

Finally,
returning to arbitrary  $X$, suppose that a point $x \in X$ has a neighborhood isomorphic to a neighborhood of a Nakajima quiver variety.
Then, under mild conditions, \cite[Theorem 1.2]{BCS22} implies that the Mori chamber structure of its crepant resolutions is given by a hyperplane arrangement coming from geometric invariant theory (GIT). Varieties that locally have this structure
include multiplicative quiver varieties and character varieties of Riemann surfaces \cite[Theorem 5.4]{KS20}, and more generally moduli of objects in 2-Calabi--Yau categories by Davison \cite[Theorem 5.11]{Davison21}, which include  Higgs bundle moduli spaces and moduli of sheaves on K3 or abelian surfaces.  In all of these cases, to classify crepant resolutions, we are reduced to determining the linear monodromies on the aforementioned hyperplane arrangements given by a set of generating exit paths.

Specializing Theorem \ref{t:constr} allows us to prove classification results in the following cases, constituting our third goal:
\begin{itemize}
    \item[(1)] varieties with each singularity locally having a unique projective crepant resolution,
    \item[(2)] symmetric powers of surfaces with du Val singularities,
    \item[(3)] certain multiplicative quiver varieties, and
    \item[(4)] finite symplectic quotients of symplectic tori.
\end{itemize}
We can relax (1) to also allow isolated singularities which may have multiple resolutions;  it remains true that locally projective crepant resolutions are classified by arbitrary choices of local projective crepant resolutions at the isolated singularities.  This includes the case where all singularities are either du Val or isolated, such as for canonical three-folds and symplectic singularities of dimension four.

Finally, in the main body of the paper, we work in the greater generality of partial crepant resolutions (not requiring that the source be smooth); note that these always exist, unlike smooth crepant resolutions.

The paper is structured as follows: 
Section \ref{s: general results} gives background results on stratified varieties and in particular establishes a large class of stratified spaces with the local product condition.
Section \ref{s: main results} contains the main results of the paper and useful corollaries.
Section \ref{s: examples} applies the main results to the examples listed above. 

\subsection{Conventions}\label{ss:conv}
We will be working with analytic varieties over the complex numbers. By a stratification, we mean a locally finite stratification by smooth, connected, locally closed subvarieties.

We denote the symmetric group on $n$ letters by $\fS_n$ and the cyclic group of order $n$ by $C_n$. For a finite set $S$ we denote its cardinality by $\#S$. We use bold to indicate categories such as $\bf{Sets}$ for the category of sets. 

For a space $X$ we use $\tilde{X}$ to denote (the source of) a resolution, and $\overline{X}$ to denote its closure. Additionally, $X^{\sing}$ denotes its singular locus and $X^{\sm}$ its smooth locus. 

We refer to a not-necessarily isolated singularity as du Val if it is analytically locally a product of a du Val singularity with a polydisc. 

\subsection*{Acknowledgments}
We would like to thank Richard Thomas for his interest, discussions, and the observation that led to Remark \ref{r: richard}.
Thanks to Paolo Cascini for useful discussions, comments on the appendix, and pointing us to \cite{Fujino22}.  We thank Osamu Fujino for answering our questions on his work on the analytic minimal model program and providing statements of  relevant consequences, and Shinnosuke Okawa for additional clarifications.  Thanks to Mirko Mauri for pointing out the example in Remark \ref{rem:example_globally_proj}, useful discussions, sharing a manuscript on terminalizations of analytically $\bQ$-factorial singularities, and carefully reading our manuscript.  We thank Pavel Safronov for pointing out the reference \cite{LNY23}. Thanks to Dmitry Kaledin, Mykola Matviichuk, Navid Nabijou, Brent Pym, Eric Rains, Michel Van den Bergh, and Michael Wemyss for stimulating conversations. Finally, the anonymous referee provided many useful suggestions and corrections.

The first-named author was supported by the Engineering and Physical Sciences Research Council (EPSRC) grant EP/S03062X/1 and subsequently by the European Research Council (ERC) grant ERC-2019-ADG. He additionally thanks the University of Birmingham, the Isaac Newton Institute, and Universiteit Hasselt where the work was carried out.

\section{Stratifications with product singularities} \label{s: general results}

Let $\mathcal{S}$ denote a stratification of $X$ by (Zariski) locally closed subvarieties. Let $S \in \mathcal{S}$ denote a stratum. 
\begin{defn}
A stratification, $\mathcal{S}$, of $X$ is \emph{locally finite} if each $x \in X$ lies in the closure of only finitely many strata.
\end{defn}
Throughout the paper we restrict to locally finite stratifications by smooth, connected, locally closed subvarieties, and simply refer to these as stratifications.   Note that, in the case $X$ is actually a complex algebraic variety (meaning, a reduced, separated scheme of finite type over $\bC$) and the stratification is algebraic, then the stratification must be finite.

We will be interested in stratifications for  which $X$ decomposes as a product along each stratum (Definition \ref{d:local product}). For such stratifications, $X$ is homogeneous along strata:

\begin{prop}\label{p:iso-germs}
Suppose that $X$ is locally a product along a connected, locally closed, smooth subvariety $S$. Then for any $s,s' \in S$, the germs of $X$ at $s$ and at $s'$ are isomorphic.
\end{prop}
\begin{proof}
Let $T \subseteq S$ be the set of all points whose germ is isomorphic to the one at $s$. This is open by the local product condition: germs at two points $(u_1, 0), (u_2, 0) \in (U \cap S) \times F \cong U$ are isomorphic given that $U \cap S$ is smooth.  But now, the complement of $T$ is also a union of open subsets, one for each isomorphism class of germs.  Since $S$ is connected, and $T$ is nonempty, $T=S$.
\end{proof}

There are many nice examples of stratifications along which $X$ is locally a product. One of the main cases we will consider is that of Poisson varieties with finitely many symplectic leaves: Recall that a variety $X$ is \emph{Poisson} if $\cO_X$ has a Lie bracket $\{-,-\}: \cO_X \times \cO_X \to \cO_X$ which is a derivation in each component. To each local section $f \in \Gamma(U,\cO_X)$ we obtain a vector field $\xi_f$ given by $\xi_f(g) = \{f,g\}$. These vector fields form an integrable distribution (local analytic foliation) since $[\xi_f, \xi_g]=\xi_{\{f,g\}}$.  A \emph{symplectic leaf} $S$ of $X$ is a leaf of this foliation, i.e., a locally closed subvariety $S$ of $X$ such that $T_s S = \text{Span}(\xi_f|_s: f \in \Gamma(U,\cO_X), \text{ $U$ a neighborhood of $x$})$; we require that $S$ be a maximal connected such subvariety. This ensures uniqueness of a symplectic leaf $S$ through $s \in X$, if it exists.

By the Weinstein splitting theorem, if $X$ is Poisson and $S \subseteq X$ is a symplectic leaf, then $X$ decomposes locally as a product along $S$, see Appendix \ref{s:appendix_Weinstein}.  

As a consequence we have the following standard result:
\begin{prop}
    If $X$ is a Poisson variety with a (locally finite) stratification by symplectic leaves, then $X$ is locally a product along this stratification.
\end{prop}

A symplectic singularity, following \cite{Beauville}, is a normal variety $X$ such that:
\begin{itemize}
    \item the smooth locus $X^{\sm}$ of $X$ has a symplectic form $\omega$;
    \item for some (equivalently, every) projective resolution of singularities $\rho: \tilde X \to X$, the pullback $\rho^* \omega$ extends to a regular two-form on $\tilde X$.
\end{itemize}
By \cite[Theorem 2.5]{Kaledin06}, every algebraic symplectic singularity has a finite stratification by symplectic leaves.

\begin{prop} \label{prop:stratification_X/G}
Let $X$ be a smooth variety and $G$ a finite group of automorphisms. Then $X/G$ has a stratification along which $X$ is locally a product.
\end{prop}

\begin{proof}
Stratify $X$ by stabilizer subgroups.  If $S$ is a stratum corresponding to a subgroup $H < G$, then at a point $s \in S$, the action of $H$ at a germ of $s$ linearizes.  
The quotient germ is equivalent to a quotient $V/K$ for $V$ a vector space and $K< \GL(V)$ a finite group of linear automorphisms. We can write $V = V^K \oplus V'$ for some complementary $K$-representation $V'$ of $V^K$.  Then $V/K \cong V^K \times V'/K$.   The stratum where the stabilizer is $K$ is $V^K$, so we indeed get a product along this. 
\end{proof}

\begin{prop}\label{p:Poisson-finite} 
Let $X$ be a Poisson variety and $G$ a finite group of Poisson automorphisms.  Let $S$ be a symplectic leaf of $X$. Then the image of $S$ in $X/G$ has a 
stratification by symplectic leaves.
\end{prop}

\begin{proof}
The proof follows in the same way as the previous one, noting now that in the case $K$ acts symplectically on $V$, then the Hamiltonian vector fields at the origin of $V/K$ span the subspace $V^K / K$.
\end{proof}

\begin{cor}\label{cor:poisson-fq}
    If $X$ is a Poisson variety with a 
    stratification by symplectic leaves, and $G$ is a finite group of Poisson automorphisms, then $X/G$ also has a 
    stratification by symplectic leaves. 
\end{cor}

In particular, this includes finite quotients of algebraic symplectic singularities, but actually these are themselves symplectic singularities by \cite[Proposition 2.4]{Beauville}.

Given an algebraic  variety $X$ with an action by a reductive group $G$ and a $G$-equivariant ample line bundle $L$, let $X^L \subseteq X$ be the semistable locus. Recall that this open subset consists of points $x \in X$ such that there is a global $G$-equivariant section of $L^{\otimes m}$ for some $m \geq 1$ which does not vanish at $x$. Given a subset $Z \subseteq X$ we will also use the notation $Z^L := X^L \cap Z$. Then $X/\!/_L G := X^L /\!/ G$.
\begin{prop} 
Let $X$ be a smooth algebraic variety and $G \times X \to X$ an action of a reductive group $G$. Assume that $X$  is equipped with a $G$-equivariant ample line bundle $L$ (one can take $L$ trivial if $X$ is affine).  
Then $X/\!/_L G$ has a finite 
stratification along which it is locally a product.
\end{prop}

\begin{proof}
It suffices to prove the result for $G$ connected, since by Proposition \ref{prop:stratification_X/G}, we can take a further quotient by any finite group. Thanks to the GIT construction, every point in $Y:= X/\!/_L G$ is the image of a closed $G$-orbit in the semistable locus of $X$, whose stabilizer is a reductive subgroup. By Luna's slice theorem, the number of conjugacy classes of such reductive subgroups is finite, which gives a finite 
stratification of $Y$ by connected components of the image of points with fixed conjugacy class of stabilizer.  

Let $S$ be such a stratum corresponding to the conjugacy class $[H]$ for $H \leq G$ reductive.  We claim that a local slice to $S$ in $Y$ is given by $V/\!/H$, for $V \subseteq X$ a slice guaranteed by Luna's theorem at a point $x \in X$ with closed $G$-orbit in the semistable locus of $X$ which maps to $s$. The action of $H$ fixes the origin $0 \in V$, and we can linearize the action so that $H < \GL(V)$. We then write $V = V^H \oplus V'$ for some $H$-subrepresentation $V'$, with $V^H$ locally giving the stratum $S$.  This implies that the germ of $Y$ at $s$ is indeed a product along $S$, as desired.
\end{proof}

\begin{prop} 
Let $X$ be an algebraic Poisson variety and $G \times X \to X$ a Hamiltonian action of a reductive group $G$. Assume that $X$  is equipped with a $G$-equivariant ample line bundle $L$ (one can take $L$ trivial if $X$ is affine).  
Let $S \subseteq X$ be an algebraic symplectic leaf.  Let $Y$ be  the reduced subscheme of $X/\!/\!/_L G$. Then the image in $Y$ of the closed $G$-orbits in $S \cap \mu^{-1}(0)^L$  are contained in finitely many algebraic symplectic leaves. 
\end{prop}

\begin{proof}  
The proof begins exactly as before: we reduce to the case that $G$ is connected (now using Proposition \ref{p:Poisson-finite}).
Since $G$ is connected, it preserves the symplectic leaf $S$.
Consider a closed $G$-orbit, $G \cdot s \subseteq S \cap \mu^{-1}(0)^L$.
Let $H\leq G$ be the stabilizer of $s$, which is reductive. Let $Z \subseteq S \cap \mu^{-1}(0)^L$ be the connected component of the locus with stabilizer conjugate to $H$ which contains $s$. The variety $Z$ is locally closed and $G$ invariant. Thanks to Luna's slice theorem, it consists only of closed $G$-orbits, and its image is a locally closed subvariety of $Y$.
 Applying \cite[Section 2.3]{Losev17}, we see that the image of $Z$ lies in a single symplectic leaf of $Y$, that is, the Hamiltonian vector fields act transitively on the tangent spaces of every point in the image.
  Finally, by Luna's slice theorem, we see that there can only be finitely many such subvarieties $Z$, since they are the connected components of loci labeled by conjugacy classes of reductive subgroups of $H$ occuring as stabilizers of points of $\mu^{-1}(0)^L$. 
  \end{proof}

\begin{cor}
If $X$ is an algebraic Poisson variety with a finite 
stratification by symplectic leaves,  $G$ is a reductive group acting Hamiltonianly on $X$, and $L$ is a $G$-equivariant ample line bundle on $X$, then the reduced subscheme of $X/\!/\!/_L G$ also has a finite 
stratification by symplectic leaves.
\end{cor}

In particular, we could take $X$ to be a symplectic singularity in the corollary.  We remark that, unlike in the case of finite groups, the quotient $X/\!/\!/_L G$ need not itself be a symplectic singularity. 

Summarizing, the following classes of varieties have a 
stratification along which $X$ is locally a product:
\begin{itemize}
    \item[(1)] any Poisson variety with finitely many symplectic leaves;
    \item[(2)] any finite or GIT quotient of a smooth variety.
\end{itemize} 
We also saw that the varieties in (1) are closed under finite or GIT Hamiltonian reductions, and that there the product decompositions are products of Poisson varieties.

Additionally, threefolds with canonical singularities $X$ admit a stratification along which $X$ is locally a product by \cite[Corollary 1.14]{Reid-c3f}.

\section{Constructibility} \label{s: main results}

The first purpose of this section is to prove Theorem \ref{t:constr}: the restriction of the sheaf $\ccR$ (defined below)
to each stratum $S$ is locally constant. We will then discuss in more detail the consequent description of isomorphism classes of locally projective crepant resolutions.
Moreover, we will relax the assumption that (the source of) a resolution is smooth, and work with partial resolutions.

\begin{defn}\label{d:ccr}
For an open set $U \subset X$, let $\ccR(U)$ denote the set of isomorphism classes of locally projective, \emph{partial} crepant resolutions of $U$ that are locally dominated by  projective terminal crepant resolutions. 
\end{defn}
Here the domination condition means that either $\tilde U \in \ccR(U)$ has terminal singularities, or else there is an open covering $U_i$ of $U$ for which the restrictions $\tilde U_i$ admit a further projective partial crepant resolution which does have terminal singularities.
Note that, under mild conditions on the singularities, \cite{BCHM10} ensures that the domination condition holds,
see Remark \ref{r:bchm}.

\subsection{Proof of Theorem \ref{t:constr}: Picard groups and product decompositions}\label{ss:pic-prod}
In this section we will prove Theorem \ref{t:constr} in the more general context of partial crepant resolutions (Definition \ref{d:ccr}).

In order to prove that $\ccR$ is constructible,
we need to study the Picard group of the local crepant resolutions using a bit of birational geometry. An important consequence is that, if a single local projective crepant resolution (or minimal model) can be obtained by resolving the slice, then \emph{all} local projective crepant resolutions (or minimal models, or partial crepant resolutions satisfying a technical condition) are obtained this way. 

\begin{defn} \label{d: prmm}
A \emph{projective (respectively locally projective) relative minimal model}\footnote{This is also called a projective (resp. locally projective) $\bQ$-factorial terminalization.} of $U$ is a projective (resp. locally projective) crepant morphism $\tilde U \to U$ with $\tilde U$ having analytically locally $\bQ$-factorial, terminal singularities. If not specified, we assume projective.
\end{defn}
Here by $\bQ$-factorial we mean that every analytic Weil divisor in $\tilde U$ has a multiple which is Cartier (hence defines a line bundle).  

Recall that our assumption on each stratum $S$ is that every point $s \in S$ has (a) a neighborhood $U$, (b) a pointed stratified variety $F \ni 0$, and (c) a stratified isomorphism $\varphi: (U \cap S) \times F \rightarrow U$. We write this data as a triple $(U, F, \varphi)$. We will restrict to a setting where the restriction map from (projective partial crepant) resolutions of $U$ to $F$ is an isomorphism. For this we will need the following key property of $F$:

\vskip 5 pt
\centerline{\emph{$F$ admits a projective relative minimal model $\tilde F \to F$ with $H^1(\tilde F, \cO)=0$. $\quad (*)$}}

\vskip 5 pt
Observe that, for $F$ Stein, the property (*) implies the same property for any pointed Stein open subset of $F$, because the restriction of a projective relative minimal model is still one, and 
$H^1(\tilde F, \cO)=0$ is
equivalent to $R^1 \rho_* \cO_{\tilde F} = 0$ for $\rho: \tilde F \to F$, so this vanishing is also inherited for Stein open subsets. So we can think of the condition (*) as a condition on the germ of a base point of each stratum.

In many situations, we will be given $\tilde F$, so that  $(*)$ holds. In particular, it holds whenever $F$ locally admits a projective crepant resolution, and this is true in all of our main examples. But  
let us take a moment to explain under what conditions $(*)$ is guaranteed to hold, according to the minimal model program.

\begin{defn}
If a variety is isomorphic to an analytic open subset of an (affine) algebraic variety, we call it ``of (affine) algebraic origin''.  A variety of algebraic origin is covered by varieties of affine algebraic origin.
\end{defn}
\begin{rem}\label{r:bchm}
    By \cite[Corollary 1.4.3, Corollary 1.3.2]{BCHM10}, when $F$ has affine algebraic origin, then there exists an algebraic relative minimal model 
    whenever $F$ has only
canonical singularities (note that this is a biconditional, since the existence of a partial crepant resolution with terminal singularities shows that the original singularities were canonical). In fact, they also show that in this case  $\tilde F$ is a relative Mori dream space over $F$, hence $F$ has only finitely many partial projective crepant resolutions dominated by terminal ones. If we assume that $X$ itself has canonical singularities
(e.g., if $X$ admits a smooth crepant resolution, or if $X$ has symplectic singularities), then the same holds for $F$, and we then only need to check the algebraicity of $F$ to apply \emph{op.~cit.}   We then obtain an algebraic minimal model. To be an analytic minimal model we need to check if it is analytically $\bQ$-factorial.  This could fail in general, but will hold in special cases. For example, if $X$ is a cone, we can choose $F$ to be a cone, and then all of the analytic Weil divisors through the cone point are linearly equivalent in a small enough neighborhood to an algebraic one, and hence all of the analytic Weil divisors in $\tilde F$ are linearly equivalent in the preimage of a neighborhood of $0$ to an algebraic one (obtained from strict transforms of divisors on $F$ and exceptional divisors).

Moreover, applying \emph{op.~cit.}~to any partial projective crepant resolution of $F$, we get that it is always dominated by a terminal one (in fact an algebraic minimal model).  So the condition of partial projective crepant resolutions being dominated by terminal ones is in fact redundant in this case.

The condition of having algebraic origin appears to be unnecessary,
by replacing results from the minimal model program used in \cite{BCHM10} by those in \cite[Theorem 1.6]{Fujino22} in the analytic setting.\footnote{Thanks to Paolo Cascini for pointing out this reference, and thanks to Osamu Fujino for multiple clarifications and sending a brief note stating the needed assertions.} Given this,  up to shrinking $F$,
$(*)$ holds precisely when $X$ has canonical singularities (which includes all symplectic singularities).
Note also that,  if $X$ has symplectic singularities, Kaledin conjectured in  \cite[Conjecture 1.8]{KaledinSurvey}
that $F$ is always conical and in particular of algebraic origin; this is true in all known examples. In this case \cite{BCHM10} will provide the needed relative minimal model.
\end{rem}

Now let us assume (*) holds. We prove the key technical result that implies Theorem \ref{t:constr}:
\begin{prop} \label{p: key} 
Suppose that $(U, F, \varphi)$ is a local product neighborhood along $S$ and $F$ satisfies $(*)$. Assume that $U \cap S$ is a polydisc (i.e., biholomorphic to a unit disc in $\bC^n$). Then every projective partial crepant resolution $\rho_U: \tilde U \to U$ which is dominated by a terminal projective resolution is isomorphic to $\Id \times \rho_F: (U\cap S) \times \tilde F \to (U\cap S) \times F$ for some projective partial resolution $\rho_F$ dominated by a terminal projective resolution. 
Moreover, restriction  induces an isomorphism $\Pic(\tilde U) \otimes_{\bZ} \bQ \to \Pic(\tilde F) \otimes_{\bZ} \bQ$ which maps the ample cone onto the ample cone.
\end{prop}
In other words, germs of partial crepant resolutions dominated by terminal ones are isomorphic to products of the stratum with a resolution of the slice, and all line bundles on the resolution are pulled back from the resolution of the slice.

For the proof, recall that a line bundle $L \in \Pic(Y)$ is \emph{movable} relative to a map $f: Y \to Y_0$ if the stable base locus has codimension at least two, i.e., outside a subset of $Y$ of codimension at least two, every point admits a nonzero value on some section $\sigma \in \Gamma(f^{-1}(U),L^n)$ for some $n \geq 1$ and some open set $U \subseteq Y_0$.  
\begin{defn}\label{d: mod-mori}
   Let $Y \to Y_0$ be a projective morphism. Given a movable bundle $L$ on $Y$ we can form a birational modification $Y(L) := \Proj_{Y_0} \bigoplus_{m \geq 0} f_* L^m
   $ over $Y_0$, together with a birational map $\varphi_L: Y \dashrightarrow Y(L)$ over $Y_0$ which is defined outside the stable base locus of $L$. We say that $L, L'$ are \emph{Mori equivalent} over $Y_0$ if $\varphi_L \circ \varphi_{L'}^{-1}$ is an isomorphism of varieties. 
\end{defn}
Typically we will have $Y_0$ Stein and $\Gamma(Y,\cO_Y) = \Gamma(Y_0,\cO_{Y_0})$.
\begin{proof} [Proof of Proposition \ref{p: key}]
Note that, since $U$ is assumed to be irreducible, so is $F$.
Take a relative minimal model $\tilde{F}$ of $F$ with $H^1(\tilde F,\cO)=0$. We then get a relative minimal model $\tilde{F} \times (U \cap S)$ of $U$. One can check that the projection map induces an isomorphism $\Pic(\tilde F \times (U \cap S)) \cong \Pic(\tilde F)$ using the long exact sequence associated to the exponential sequence and K\"{u}nneth formula together with: $H^1(\tilde F,\cO)=0=H^1(U \cap S,\cO)$ and $H_{\text{dR}}^{>0}(U \cap S)=0$.  This pullback clearly preserves the ample cones, and this yields the final statement.

The isomorphism of Picard groups restricts to an identification of movable cones and Mori decompositions. Next, note that any two projective crepant birational morphisms $\tilde U_1, \tilde U_2 \to U$ from  terminal irreducible varieties $\tilde U_1,\tilde U_2$ are isomorphic in codimension one \cite[Theorem~3.52]{KollarMori} (essentially, the crepancy and terminality imply  the two have to extract the same divisors). 
Thus the Weil divisor class groups of $\tilde U_1$ and $\tilde U_2$ are identified. If $\tilde U_1$ is $\bQ$-factorial, then this gives an identification $\Pic \tilde U_2 \otimes_{\bZ} \bQ \subseteq \Pic \tilde U_1 \otimes_{\bZ} \bQ$. 

In this case, letting $L$ be a relatively ample bundle for $\tilde U_2 \to U$, then $L$ is movable for $\tilde U_1$, and we get
a birational map $\tilde U_1 \dashrightarrow \tilde U_2 \cong \tilde U_1(L)$ (working over $U$).
If, instead, $\tilde U_2 \to U$ is not necessarily terminal, but is a partial crepant resolution 
which is dominated by a terminal one $V \to \tilde U_2 \to U$, then $\tilde U_2$  is obtained from $V$ by 
 a birational morphism defined by  a line bundle in the nef cone of $V$, which is well known to lie in the closure of the movable cone. In view of the fact that all line bundles on our minimal model $\tilde F \times (U \cap S) \to U$ are pulled back from one on a minimal model $\tilde F \to F$, we see that all the other partial crepant resolutions are isomorphic to the product of $U \cap S$ and the modification of $\tilde F$ defined by the appropriate line bundle. This yields the first statement. 
\end{proof}


For the next corollary, let $\ccR^F$ denote the sheaf on $F$ defined in the same way as $\ccR$.
\begin{cor}\label{c:local-product-res}
Given a local product neighborhood $(U, F, \varphi)$, taking products of resolutions with $U \cap S$ yields an isomorphism of 
the restriction of $\ccR$ to $U \cap S$ with the constant sheaf with stalk equal to the stalk of $\ccR^F$ at $0$. 
\end{cor}
\begin{proof}
    The restriction $\ccR|_{U \cap S}$ takes value on an open subset $V$ given by the limit of $\ccR(U')$ over open neighborhoods of $V$ in $X$. We can take a basis of these to have the form $V \times F'$ for $F' \subseteq F$ containing $0$. Then, the proof follows from Proposition \ref{p: key}.
\end{proof}
 In particular, $\ccR$ is locally constant on $S$, proving Theorem \ref{t:constr}.

\begin{rem}\label{r: richard} \footnote{Thanks to Richard Thomas for the question that led to this remark.}
    Given a path $\gamma: [0,1] \to S$, using local product neighborhoods about points of the path, we can find a non-canonical isomorphism of the germ of $X$ at $\gamma(0)$ and the germ of $X$ at $\gamma(1)$, as in Proposition \ref{p:iso-germs}. Then one can check that the monodromy about the path is given by composing resolutions by this automorphism of the germ.  This gives a restriction on the monodromy action: it can only relate partial crepant resolutions which differ by an automorphism of the germ of the base. In particular, the two local partial crepant resolutions are isomorphic as abstract varieties (at least, after restricting to suitable open neighborhoods of the singularity). 
    
    In more detail, postcomposition with the automorphism group $G_s$ of the germ of $X$ at $s \in S$ gives an action on the set $\mathcal{R}_s$ of local partial crepant resolutions at $s$. This action is determined by the action of $G_s$ on the discrete Picard group of the germ, which 
    factors through $\pi_0(G_s)$. 
    The above shows that every monodromy automorphism of $\mathcal{R}_s$ 
    is obtainable from this action. 
\end{rem}

\subsection{The constructible (pre)sheaf of relative Picard groups: Theorem \ref{t:pic}}

Let $X$ be an algebraic variety with canonical divisor $K_X$. Let $\pi: \tilde{X} \rightarrow X$ be a morphism of algebraic varieties. Recall that the relative Picard group is defined as $\Pic(\tilde X/X):=\Pic(\tilde X)/\pi^* \Pic(X)$.
We now define the presheaf $\cP$ of relative Picard classes without requiring the existence of a smooth crepant resolution:
\begin{prop} Let $X$ be a normal variety with $K_X$ $\bQ$-Cartier.
There is a canonical presheaf $\cP : \textbf{Op}(X) \rightarrow \Vect_{\bQ}$,
\begin{equation}
 \cP(U) = \begin{cases} \Pic(\tilde{U}/U) \otimes_{\bZ} \bQ, &\text{for a relative minimal model $\tilde{U} \to U$}, \\ 0, &\text{if there is no relative minimal model,}
 \end{cases}
\end{equation}
which is well-defined up to canonical isomorphism.
\end{prop}
\begin{proof}
Any two choices $\tilde{U}_1$ and $\tilde{U}_2$ of relative minimal models have canonically isomorphic relative rational Picard groups $\Pic(\tilde{U}_1/U)\otimes_{\bZ} \bQ \cong \Pic(\tilde{U}_2/U)\otimes_{\bZ} \bQ$, as explained in the proof of Proposition \ref{p: key}:  the two resolutions are canonically isomorphic in codimension one by \cite[Theorem 3.52]{KollarMori}, so have the same Weil divisor class groups, and then the fact that they are analytically locally $\bQ$-factorial gives the statement. By restricting relative minimal models and line bundles, this forms a presheaf.
\end{proof}
Now the proof of Theorem \ref{t:constr} yields the following:
\begin{thm}\label{t:P-constr}
    Let $X$ have a 
    stratification satisfying (*) of Section \ref{ss:pic-prod}. Then,
   the sheafification $\cP^\sharp$ of $\cP$ is constructible along the stratification. 
\end{thm}
\begin{proof}
If $(U \cap S) \times F$ is a local product neighborhood of $s$, then  for any $s' \in U \cap S$, and any neighborhood $U' \subseteq (U \cap S) \times F$ of $s'$, we can restrict to a smaller neighborhood of the form $(U' \cap S) \times F' \subseteq U' \subseteq (U \cap S) \times F$. Then given a minimal model $\tilde F' \to F'$ we obtain $\cP((U \cap S) \times F') \cong \cP(F')$, by Proposition \ref{p: key}.  So we see that the stalks of $\cP$ at all $s' \in U \cap S$  are equal to the stalk of $\cP|_{F}$ at the basepoint $0 \in F$. In particular the sheafification of $\cP$ is locally constant with stalk equal to the latter stalk.
 \end{proof}

 \begin{cor}
 The sheaf $\ccR$ is the sheafification of the presheaf $U \mapsto \overline{\Mov(\cP_U)}/\!\!\sim$, where $\Mov(\cP_U)$ is the cone of relative classes of movable bundles on $\tilde U/U$, and $\sim$ is the Mori equivalence relation.   
 \end{cor}
 \begin{proof}
    Given a local section $\sigma \in \Gamma(U,\ccR)$, first shrink $U$ so that $\tilde U \to  U$ is either terminal or dominated by a terminal resolution.  Letting $\tilde U \to U$ be a relative minimal model, we can take a relatively ample bundle to get an associated class in $\Pic(\tilde U/U)$; the Mori equivalence class of this determines a local section of $\overline{\Mov(\cP_U)}/\!\!\sim$. (This section is in the movable cone if $\sigma$ is terminal, and in the boundary if $\sigma$ is merely dominated a terminal resolution.) 
    
    Conversely, given a local section  $L \in \overline{\Mov(\cP_U)}/\!\!\sim$ we get by the Proj construction $\tilde U(L)$ a local projective partial crepant resolution in $\ccR(U)$. Now, the condition for local sections to be compatible is that they restrict to isomorphic resolutions on overlaps, which is equivalent to having sections that are Mori equivalent on the overlap.  Note that in the case of $\ccR$,  compatibility implies they glue to a global (locally projective partial crepant) resolution, hence a section of $\ccR$. But for $\overline{\Mov(\cP_U)}/\!\!\sim$ local sections only glue in general to a section of \emph{the sheafification}. (To make sense of a global section of the original presheaf, the locally projective resolution would have to be globally projective, and we would need to have a globally defined relative minimal model).
 \end{proof}
 As a result, given a section $\sigma \in \Gamma(X, \cP^\sharp)$ whose values are locally in $\overline{\Mov(\cP_U)}$, we get an associated partial crepant resolution $\tilde X_\sigma \in \ccR(X)$. Conversely, given $\tilde X \in \ccR(X)$, we define an associated section $\sigma \in \Gamma(X,\cP^{\sharp})$ to be one which is locally the class of a relatively ample bundle on $\tilde X$.

 \begin{rem} \label{r:absolute-relative}
     For a contractible Stein neighborhood $U \subseteq X$, we have $H^1(U, \cO)=0=H^2(U,\bZ)$. So the exponential sequence gives $\Pic(U) \cong H^1(U,\cO^\times)=0$.  By \cite{Gilmartin-loc-contr}, see also \cite[2.10]{Milnor-singular-hypersurfaces}, every analytic variety has a basis of such neighborhoods.  Hence, $\cP^\sharp$ is also the sheafification of the presheaf of absolute (not relative) Picard groups. 
     However, $\cP^\sharp$ is in general much closer to $\cP$ than to the presheaf of absolute groups; for instance, $\cP^\sharp$ is unaffected by the presence of nontrivial  line bundles on $X$, unlike the presheaf of absolute Picard groups.
 \end{rem}
 
 \subsection{Criteria for locally projective crepant resolutions to be globally projective}
 Given a locally projective partial crepant resolution $\tilde X \to X$, note that $\tilde X$ is globally projective if and only if there exists a relatively ample line bundle on $\tilde X$.  Such a bundle would give rise to a global section of $\cP^\sharp$ which lies in the section of Mori equivalence classes of cones determined by the global section of $\ccR$ defining $\tilde X$.  We obtain:
\begin{cor} \label{c:glob-obstr}
Given a locally projective partial crepant resolution $\tilde X \to X$, the resolution is globally projective  only if there exists a global section of $\cP^\sharp$ lying in the section of Mori equivalence classes determined by $\tilde X$.
\end{cor}
\begin{rem} \label{rem:example_globally_proj} In general, the condition here is not sufficient.  For example\footnote{Thanks to Mirko Mauri for pointing out this example.}, let $X \subseteq \bP^4$ be a hypersurface of degree $d > 2$ with at most $\frac{(d-1)^2}{4}$ ordinary double points.  Then by \cite{Cheltsov-factoriality}, $X$ is $\bQ$-factorial (Cheltsov is working in the algebraic category, but by GAGA the same statement holds in the analytic category). Moreover, $X$  admits small (hence crepant) locally projective resolutions $\rho: \tilde X \to X$, obtained by blowing up each singularity, yielding $\bP^1 \times \bP^1$, then contracting one of the rulings (there are two choices of local resolution at each singularity). We claim that $\tilde X$ is not globally projective.  

To see this, since the resolution is small, the exceptional locus $\rho^{-1}(X^{\text{sing}})$ has codimension at least two. Let $\Cl$ denote the Weil divisor class group. 
We obtain 
\[
\Pic(\tilde X) \otimes \bQ = \Pic(\tilde X \setminus \rho^{-1}(X^{\text{sing}})) \otimes \bQ = \Cl(X \setminus X^{\text{sing}}) \otimes \bQ = \Pic(X) \otimes \bQ,
\]
the last equality following from $\bQ$-factoriality of $X$. That is, the relative rational Picard group is trivial.  So there can be no relatively ample global algebraic line bundle, as the map $\tilde X \to X$ is not an isomorphism.

Furthermore,  a global section $\sigma \in \Gamma(X,\cP^\sharp)$ associated to $\tilde X$ is equivalent to a local choice, at each singularity, of relative ample class, which are all independent; in particular such $\sigma$ exist.
\end{rem}
We strengthen Corollary \ref{c:glob-obstr} to a biconditional giving the obstruction for a locally projective partial crepant resolution to be globally projective. For this it is necessary to work with an integral
version of $\cP^\sharp$:
\begin{defn}
Given $\rho: \tilde X \to X$ in $\ccR(X)$, we let $\cP^{\tilde X}_\bZ$ be the presheaf on $X$ assigning to open subsets $U \subseteq X$ the integral relative Picard group, $\cP^{\tilde X}_\bZ(U)=\Pic(\rho^{-1}(U)/U)$, and let $\cP^{\tilde X,\sharp}_\bZ$ be its sheafification.
\end{defn}
Note that, in order to define this presheaf in the integral context, we need to fix a resolution $\tilde X$ in $\ccR(X)$, because two different local minimal models only canonically have the same rational Picard group.

Pullback of local relative classes defines a canonical map $\cP^{\tilde X,\sharp}_\bZ \to \cP^\sharp$. We call a preimage of $\sigma$ under this map an integral lift.
\begin{rem}
There is no guarantee that a given section $\sigma \in \Gamma(X, \cP^\sharp)$ defining a partial resolution $\tilde X \in \ccR(X)$ admits an integral lift. However, under mild conditions, some multiple does. Suppose that the stratification of $X$ is finite (e.g., if $X$ is an algebraic variety with algebraic stratification).  Then we can take a multiple of $\sigma$ such that over some local product neighborhood at each stratum, $N\sigma$ is actually represented by a line bundle. If furthermore $m\geq 1$ kills the torsion in the relative Picard groups over these local product neighborhoods, then $mN\sigma$ has a canonical integral lift (given by $m$ times an arbitrary choice of local integral lift of $N\sigma$). It follows 
that these canonical integral lifts extend to a global integral lift of $mN\sigma$.
\end{rem}
\begin{rem}
$\cP^{\tilde X, \sharp}_\bZ$ is also constructible thanks to Corollary \ref{c:local-product-res}.
\end{rem}
\begin{thm}\label{t:locproj-gone-global}
Given a section $\sigma \in \Gamma(X,\cP^\sharp)$ which defines a locally projective partial crepant resolution $\tilde X \in \ccR(X)$,  and an integral lift $\tilde \sigma
\in \Gamma(X,\cP_{\bZ}^{\tilde X,\sharp})$, there is a canonical gerbe $\ccL_{\tilde \sigma}$ on $X$ which is trivial if and only if there exists a relatively ample line bundle on $\tilde X$ with associated local relative  class $\tilde \sigma$. Such a bundle is unique up to tensoring by the pullback of a line bundle on $X$.
\end{thm}
This in particular 
implies Theorem \ref{t:gerbe-intro}.
\begin{proof}[Proof of Theorem \ref{t:locproj-gone-global}] Let $\rho: \tilde X \to X$ be the resolution map.
The gerbe $\ccL_{\tilde \sigma}$ is defined as follows: 
\[\ccL_{\tilde \sigma}(U) := \{\text{line bundles on $\rho^{-1}(U)$ 
in the relative class $\tilde \sigma$}\}.\]
This clearly defines a presheaf of categories on $X$. In fact it defines a full substack of the stack of line bundles on $\tilde X$, pushed forward to $X$. In particular $\ccL_{\tilde \sigma}$ is a stack.

Next, the category of line bundles on $X$ acts on $\ccL_{\tilde \sigma}$ by tensoring by their pullbacks. We claim that $\ccL_{\tilde \sigma}$ is locally isomorphic to the category of line bundles with the usual tensor product.  To see this, let $U \subseteq X$ be an open subset with preimage $\tilde U \subseteq \tilde X$ admitting a relatively ample line bundle $L$.  Let $\rho_U: \widetilde U \to U$ be the projection. Then we need to check that $\Hom_{\widetilde U}(L \otimes \rho_U^* E, L \otimes \rho_U^* F) \cong \Hom_U(E,F)$. This reduces to the statement $(\rho_U)_* \rho_U^*(E^\vee \otimes F) \cong E^\vee \otimes F$. Applying the projection formula, this in turn reduces to  $(\rho_U)_* \cO_{\widetilde{U}} \cong \cO_U$. This statement is true if the Stein factorization of $\rho_U$ is trivial, which it is because $\rho_U$ is proper and birational and  $U$ is normal (hence every finite birational map to $U$ is an isomorphism, by Zariski's main theorem).

Thus, $\ccL_{\tilde \sigma}$ defines a gerbe.  The gerbe is trivial if and only if there is an equivalence of categories with the category of all line bundles on $X$. Such an equivalence gives rise to a global section of $\ccL_{\tilde \sigma}$, defined as the image of the trivial line bundle.  This gives rise to a globally defined line bundle which is in the class $\tilde \sigma$. Conversely, such a globally defined line bundle defines a trivialization of the gerbe, by sending this line bundle to the trivial line bundle on $X$. By definition of the gerbe, this line bundle is unique up to tensoring by line bundles pulled back from $X$.
\end{proof}
We immediately deduce the following:
\begin{cor}\label{c:glob-bicond}
A 
resolution $\tilde X \to X$ in $\ccR(X)$ is globally projective if and only if there exists a section $\sigma \in \Gamma(X,\cP^\sharp)$ lying in the associated Mori equivalence classes of relatively ample bundles with an integral lift $\tilde \sigma$ whose associated gerbe $\ccL_{\tilde \sigma}$ is trivial.
\end{cor}
In the remainder of this section, we will give criteria for the gerbes $\ccL_{\tilde \sigma}$ to be trivial and hence for the converse of Corollary \ref{c:glob-obstr} to hold. 

The gerbe $\ccL_{\tilde \sigma}$ gives rise to a class in $H^2(X,\cO^\times)$, which is trivial if and only if $\ccL_{\tilde \sigma}$ is trivial. In particular, we have, using the exponential exact sequence:
\begin{cor}
    If $H^2(X,\cO^\times)$ is trivial, e.g., when $H^2(X,\cO) = 0 = H^3(X,\bZ)$, then $\ccL_{\tilde \sigma}$ is trivial for all $\tilde \sigma$, so a locally projective partial crepant resolution $\tilde X$ is globally projective if and only if there exists an associated section $\tilde \sigma \in \Gamma(X,\cP^{\tilde X,\sharp}_\bZ)$.
\end{cor}

Another way to approach triviality of the gerbe $\ccL_\sigma$ is by directly considering the restriction map from global divisors. It is useful to consider this more generally for terminal (not necessarily $\bQ$-factorial) resolutions. 
\begin{prop}
For any terminal $Y \in \ccR(X)$, we have a canonical restriction map
  \begin{equation}\label{e:rest-cl}
r: \Cl(Y) \to \Gamma(X, \cP^\sharp),
\end{equation}
and for two different terminal $Y_1, Y_2 \in \ccR(X)$, the maps $r_1, r_2$ are identified via the canonical isomorphism $\Cl(Y_1) \cong \Cl(Y_2)$.
\end{prop}
Note that, by \cite[Corollary 1.4.3]{BCHM10} (see also Remark \ref{r:bchm}) a terminal resolution $Y \in \ccR(X)$ exists when $X$ is a normal quasi-projective algebraic variety with canonical singularities. Whenever a terminal $Y \in \ccR(X)$ exists,  then we can define the map $r$ in the Proposition, which up to canonical isomorphism does not depend on $Y$.
\begin{proof}
We begin with the last statement. Given 
two terminal projective partial crepant resolutions $\tilde U_1, \tilde U_2 \to U$ of a singularity $U$, as noted in the proof of Proposition \ref{p: key}, $\tilde U_1$ and $\tilde U_2$ coincide outside subsets of codimension at least two, and hence their class groups of Weil divisors modulo linear equivalence are canonically identified, in a way compatible with restriction to open subsets. 

We next explain why $r$ is well-defined given $Y$.  Although the restriction map  produces Weil divisors on the local relative minimal models, these are all $\bQ$-Cartier by $\bQ$-factoriality of the minimal model. Since they come from a global Weil divisor, they will be compatible and glue to a section of $\Gamma(X, \cP^\sharp)$.
\end{proof}
\begin{prop}\label{p:r-surjective}
    Assume that the stratification on $X$ is finite, and that a terminal $Y \in \ccR(X)$ exists with surjective map $r$ in \eqref{e:rest-cl}. Then for any $\sigma \in \Gamma(X,\cP^\sharp)$ locally in the closure of the movable cone, with associated resolution $\tilde X_\sigma$ dominated by a terminal $Y \in \ccR(X)$, 
    $\tilde X_\sigma$ is globally projective with relatively ample rational class $\sigma$. 
    \end{prop}
    More precisely, the conclusion states that, for some $N\geq 1$,  $N\sigma$ is the local rational relative class of a relatively ample line bundle for $\tilde X_\sigma$. 
 \begin{proof}[Proof of Proposition \ref{p:r-surjective}]
     For $\sigma, \tilde X_\sigma, Y$ as in the statement,  there exists a Weil divisor $D \subseteq Y$ such that $D$ is locally over $X$ equivalent to the pullback of a relatively ample $\bQ$-Cartier divisor  on $\tilde X_\sigma$. Because the stratification is finite, there is some $N\geq 1$ such that $ND$ is Cartier locally at a point of every stratum. Thanks to Corollary \ref{c:local-product-res}, this means that $ND$ is Cartier on $Y$. As in the proof of Theorem \ref{t:locproj-gone-global}, the direct image of $ND$ to $\tilde X_\sigma$ is a relatively ample Cartier divisor for $\tilde X_\sigma \to X$, and hence $\tilde X_\sigma$ is globally projective.
 \end{proof}
\begin{rem}
    Note that, by definition, all partial resolutions in $\ccR(X)$ are locally dominated by terminal resolutions. But they need not be globally so dominated (a priori at least). This explains why this condition in the proposition is nontrivial.  As we already explained though, in many cases (e.g., algebraic) it is guaranteed that every $X$ (and every partial resolution thereof) is dominated by a terminal resolution.
\end{rem}

We now give a stratumwise criterion for the map $r$ to be surjective. It is useful for a section $\sigma$ of $\cP^\sharp$ to call its \emph{support} the locus of $x$ where, on every local partial crepant resolution, $\sigma$ is nonzero, i.e., represented by a line bundle not pulled back from $X$.  
\begin{lem}\label{l:glob-crit}
Suppose that there is a terminal resolution $\tilde X \in \ccR(X)$, and that, 
for every stratum $S$, every  section $\sigma \in \Gamma(S,\cP^\sharp|_S)$ supported on $S$ is in the image of $r$.
Then $r$ is surjective.
\end{lem}
\begin{proof}
We first assume the stratification is finite and later we explain how to remove this assumption.

For a contradiction, suppose $\sigma \in \Gamma(X, \cP^\sharp)$ is a section with minimal support not of the form $r(D)$. If $S$ is a stratum which is open in the support of $\sigma$, we can find a global
Weil divisor $D$ on some terminal $\tilde X \in \ccR(X)$ supported on the closure of the preimage of $S$ such that $\sigma|_S = r(D)|_{S}$. Then $\sigma' := \sigma-r(D) \in \Gamma(X,\cP^\sharp)$ has  smaller support, so by assumption, there exists another terminal $\tilde X'\in \ccR(X)$ and $D' \subseteq \tilde X'$ with $\sigma'=r(D')$. But up to linear equivalence, $D'$ induces a canonical divisor on $\tilde X$, so we can assume $\tilde X= \tilde X'$. Then $\sigma = r(D) + r(D') = r(D+D')$, a contradiction. 

We used finiteness of the stratification for the existence of a section with minimal support not in the image of $r$. This hypothesis can be eliminated: instead, for any $\sigma$, we can find a global Weil divisor on some terminal $\tilde X$ supported over the closure of every open stratum of the support of $\sigma$. We can assume all the $\tilde X$ are the same since terminal crepant resolutions coincide outside codimension one. We can take the union of the obtained Weil divisors, which will be a Weil divisor since only finitely many Weil divisors can intersect at every point of $\tilde X$, the ones supported on closed strata meeting at the image of the point in $X$. Subtracting the image of the resulting Weil divisor under $r$ from $\sigma$ will decrease the dimension of the support of $\sigma$, so induction on dimension establishes the desired result. \qedhere


\end{proof}
\begin{rem}
    If the assumption of the lemma holds for every \emph{local} section $\sigma \in \cP_s, s \in S$, then there is no monodromy of $\cP^\sharp$ on $S$.
To see this, every local section $\sigma_s$ is linearly equivalent in a neighborhood of $s$ to $r(D)$ for a global Weil divisor $D$ on a terminal  $\tilde X \in \ccR(X)$. But then $r(D)$ provides an extension of $\sigma_s$ to a section of $\cP^{\sharp}|_{S}$. As $\sigma_s$ was arbitrary, we conclude that $\cP^{\sharp}$ has no monodromy on. $S$. 
\end{rem}


\subsubsection{Symplectic singularities}
In the case that $X$ is a symplectic singularity, we can say more. First of all, we recall a description of the exceptional divisors in this case:
\begin{prop}\cite{NamikawaNote},\cite[Corollary 0.2]{Namikawa-nilpotentcover} \label{p:ssing-edivs} If $X$ is an irreducible symplectic singularity and $f: Y \to X$ is a projective partial crepant resolution, then for every exceptional divisor $D \subseteq Y$, the image $f(D)$ has codimension two. 
\end{prop}
As a consequence, if $f: Y \to X$ is only locally projective, and $D \subseteq Y$ is an exceptional divisor, then the image $f(D)$ locally has codimension two, hence it has codimension two.

We will now show that the conditions of Lemma \ref{l:glob-crit} hold for symplectic quotient singularities, hence the converse to Corollary \ref{c:glob-obstr} holds:\footnote{Thanks to Mirko Mauri for correspondence 
inspiring this statement.} 
\begin{thm}\label{t: quotient-sing-globally-projective}
Suppose that $X$ is an analytically locally $\bQ$-factorial symplectic singularity, e.g., has only symplectic quotient singularities, and that its symplectic stratification is finite. Then \eqref{e:rest-cl} is surjective.
Thus,  a terminal $\tilde X \in \ccR(X)$
is globally projective
if and only if a compatible section $\sigma \in \Gamma(X,\cP^\sharp)$ of relative ample classes exists.
\end{thm}  This in particular implies Theorem \ref{t:qfact-symp-intro}. To prove it,
first we have the following basic lemma:
\begin{lem}\label{l:qfact-divs}
    If $X$ is analytically locally $\bQ$-factorial, then for every open subset $U \subseteq X$, $\cP(U)$ is spanned by the classes of 
    exceptional divisors. 
\end{lem}
\begin{proof} Given that $X$ is analytically  locally $\bQ$-factorial, then for every $x \in X$, neighborhood $U \ni x$, and partial crepant resolution $\rho: \tilde U \to U$ with  irreducible exceptional divisors $D_i$, the Weil divisor class group $\Cl(\tilde U)$ is spanned by the exceptional divisor components $[D_i]$ together with $\Cl(\tilde U \setminus \bigcap D_i)$. But $\tilde U \setminus \bigcap D_i$ coincides outside codimension two with the locus $U^\circ \subseteq U$ over which $\rho$ is an isomorphism, which by normality of $U$ coincides with $U$ itself.  So $\Cl(\tilde U)$ is generated by $\Cl(U)$ and $[D_i]$. By analytically local $\bQ$-factoriality we have $\Cl(U)=\Pic(U)$. Thus, every Weil divisor on $\tilde U$ is linearly equivalent to a linear combination of exceptional divisor components plus the pullback of a line bundle on $U$. As a result, $\cP(U)$ is spanned by exceptional  divisor components.
\end{proof}
\begin{proof}[Proof of Theorem \ref{t: quotient-sing-globally-projective}]
 In view of Lemma \ref{l:qfact-divs}, if $S$ is a stratum and $\sigma \in \Gamma(S,\cP^\sharp|_S)$, and for some open $U$, $\sigma|_U$ is supported on $U \cap \overline{S}$, then $\sigma|_U$ is spanned by exceptional divisors in $\tilde U \in \ccR(U)$ whose image is $U \cap \overline{S}$.  
By Proposition \ref{p:ssing-edivs}, $S$ must have codimension two. In view of Lemma \ref{l:glob-crit} and Proposition \ref{p:r-surjective}, it thus suffices to show that for every codimension two stratum $S$, the space of such sections $\sigma$ is spanned by the exceptional divisors $D$ in terminal locally projective resolutions $\rho: \tilde X \to X$ with image $\rho(D)=\overline{S}$. 

On $S$, the sheaf $\cP^\sharp$ is locally described as follows: Let $s \in S$ and let $U \ni s$ be a local product neighborhood, $U \cong (U \cap S) \times F$. Then $F$ is a surface with unique singularity the basepoint $0 \in U$. By construction $F$ can be taken to be a symplectic singularity, hence a du Val singularity. We can choose $U$ small enough that $F$ is isomorphic to a neighborhood of zero of $\bC^2/H$ for some finite $H < \SL_2(\bC)$. Let $\tilde F \to F$ be the minimal resolution.  Then $\cP(U) \cong \cP(F) \cong \Pic(\tilde F) \otimes_\bZ \bQ \cong H^2(\tilde F, \bQ)$. 

Globalizing this over $S$, take a neighborhood $U$ of $S$ which does not intersect any other strata (e.g., the union of the open stratum and $S$).  Let $\rho: \tilde U \to U$ be the minimal resolution. Then the previous paragraph implies that $\cP^\sharp \cong \rho_* \bQ$.  By \cite[Proposition 4.2]{Namikawa11}, we see that $\Gamma(S, \cP^\sharp|_{S})$ has dimension equal to the number of irreducible exceptional divisors.  Since the exceptional divisors locally span $\cP|_S$, this implies that the global irreducible exceptional divisors span $\Gamma(S,\cP^\sharp|_S)$.
\end{proof}
Since $\cP^\sharp$ is constructible by Theorem \ref{t:P-constr}, this reduces the study of globally projective terminal crepant resolutions of these singularities to combinatorics, in the same way as Theorem \ref{t:constr} does for locally projective partial crepant resolutions.  

\begin{exam}
    When $X$ is a four-dimensional symplectic singularity,  locally projective terminal crepant resolutions are given by gluing the minimal resolution of the two-dimensional leaves to arbitrary choices of projective terminal crepant resolutions of neighborhoods of the zero-dimensional leaves.  If $X$ is moreover analytically locally $\bQ$-factorial, e.g., a quotient singularity (which needs to be checked only at the zero-dimensional leaves), then Theorem \ref{t: quotient-sing-globally-projective} 
    asserts that these resolutions are globally projective if and only if, for each two-dimensional leaf $S$, the images of the Mori cones for the resolutions of the zero-dimensional leaves on its boundary all overlap and, if $\cP^\sharp$ has monodromy along $S$ (equivalently, the Namikawa Weyl group is a folding of the ADE Weyl group for the slice along $S$),  the intersection of the images of the Mori cones contains an invariant for the monodromy.
\end{exam}
\begin{exam}    For a special example, suppose $X = Y / T$ where $Y$ is a smooth connected symplectic fourfold and $T$ is the binary tetrahedral group, acting faithfully. Suppose that $T$ acts symplectically such that the zero-dimensional leaves of $Y/T$ are the images of the fixed points of $T$ where we get  the rank-two Shephard--Todd complex reflection group $G_4$ acting now symplectically on a four-dimensional space. The two-dimensional leaves have isotropy group $\bZ_3$, giving $A_2$ singularities, and each zero-dimensional leaf is on the boundary of a unique such leaf. 
    By \cite{Bellamy16}, 
    there are two projective (smooth) crepant resolutions of $\bC^4/G_4$, corresponding to dividing  the ample cone  of the $A_2$ singularity into two canonical pieces, with dividing line the fixed locus of the symmetry of the $A_2$ diagram. By the preceding paragraph, the locally projective crepant resolution is globally projective if and only if each two-dimensional leaf whose closure contains some zero-dimensional leaves has no monodromy of $\cP^\sharp$ (i.e., there are two exceptional divisors over the two-dimensional leaf, i.e., the leaf contributes a Namikawa Weyl group of $\fS_3$) and the local resolutions at the zero-dimensional leaves all pick out the same Mori cone.  So the number of globally projective crepant resolutions is $2^m$ where $m$ is the number of two-dimensional leaves ($A_2$ singularities) which contain zero-dimensional leaves, provided again that these leaves have no monodromy (otherwise there is no projective crepant resolution).
\end{exam}
\subsection{Exit paths}\label{ss: exit paths}
Given a constructible sheaf, it easily follows that parallel transport can be extended to exit paths, and that global sections are the same as exit-path compatible sections (a weak form of MacPherson's equivalence):
\begin{prop}
    Let $\cF$ be a constructible sheaf on a stratified variety $(X,\cS)$.
    Given an exit path $\gamma: [0,1] \to X$, there is a unique parallel transport map $\gamma_*: \cF_{\gamma(0)} \to \cF_{\gamma(1)}$ such that, for each $s \in \cF_{\gamma(0)}$, there is an open covering of $[0,1]$ with compatible local sections restricting at $\gamma(0), \gamma(1)$ to $s$ and $\gamma_*(s)$, respectively. Moreover, a global section of $\cF$ on $X$ is the same as an element of $\cF_x$ for all $x \in X$ compatible with this parallel transport for all exit paths. 
    \end{prop}
    \begin{proof}
To define parallel transport along a simplified exit path $\gamma: [0,1] \to X$, we first note that for each local section $s \in \cF_{\gamma(0)}$, there is some $\varepsilon$ for which the section is defined on $\gamma[0,\varepsilon]$. We can then extend to the whole path using parallel transport along the stratum.  Since all exit paths are compositions of simplified exit paths, this defines parallel transport for all exit paths.  To see that the parallel transport is unique, note that the parallel transport along a given stratum is unique, as is restriction of a local section to a smaller section. Now if we have a section over $\gamma([0,\varepsilon])$ and restrict it to a section over $\gamma([0,\delta])$ with $0 < \delta < \varepsilon$, the parallel transport along $[\delta,\varepsilon]$ recovers the same section as the latter parallel transport is unique. 

A global section of a constructible sheaf on $(X,\cS)$ is given by a collection of sections $f_i$ on neighborhoods $U_i$ of each stratum $S_i$ which agree on overlaps.  We can choose $U_i$ so that $U_i$ intersects only those strata $S_j$ such that $S_i \subseteq \overline{S_j}$ (note that by the local product condition, this is equivalent to $S_i \cap \overline{S_j} \neq \emptyset$).  The sections on each stratum are the same as choices of stalks at each point compatible with parallel transport along the stratum (by the equivalence between local systems and functors from the fundamental groupoid to sets). These sections are compatible if and only if they restrict to the same local section in some neighborhood of each point $s$ in the overlap. The latter can be checked by compatibility with an exit path ending at $s$, in view of the uniqueness of the parallel transport. It follows that any choice of local sections compatible under exit paths glues to form a global section. The converse is clear from uniqueness of parallel transport.
\end{proof}
    
We can moreover restrict to simplified exit paths, since all exit paths are compositions of these. We can be more restrictive and check compatibility only with certain paths:

\begin{prop} \label{p:compatibility-limited-exitpaths} 
Let $(X,\cS)$ be a 
stratified variety by strata satisfying Definition \ref{d:local product}. Compatibility of a collection $(f_x \in \cF_x)$ of elements of stalks with all exit paths can be checked on simplified exit paths beginning at a fixed basepoint $s_i$ of each stratum $S_i$. Moreover we can restrict to exit paths of the following form, for each stratum $S_i$: 
    \begin{enumerate}
        \item For each $q \in S_i$, a single path from $s_i$ to $q$;
        \item A set of closed paths generating  the fundamental group $\pi_1(S_i, s_i)$;
        \item For each stratum $S_j \neq S_i$ such that $\overline{S_j} \supseteq S_i$, any choice of neighborhood $U_i$ of $s_i$ in $X$, and each component  of $U_i \cap S_j$, a single simplified exit path from $s_i$ to the given component.
    \end{enumerate}
\end{prop}

\begin{proof}
    First, using paths in (1) and (2) we can obtain all homotopy classes of paths in the stratum $S_i$. It is standard that parallel transport for a local system along homotopic paths is the same, so these are enough to guarantee we obtain a global section on $S_i$. 
    
    It remains to show that the sections defined on each stratum are compatible using (3). Every simplified exit path can be written as a concatenation of such paths and exit paths which lie in an arbitrarily small ball about the initial point.  So all we need to show is that we can take this initial point to be a fixed basepoint $s_i$ of its stratum $S_i$.  Let $q \in S$ be another point of the same stratum, and let $\gamma$ be a path in $S$ from $s_i$ to $q$.  
    
    We can find an open covering $\mathcal{U}$ in $X$ of the stratum $S_i$ by open product neighborhoods such that $\cF|_{U \cap S_i}$ is constant for each $U \in \mathcal{U}$ (since each point of $S_i$ is contained in such a $U$). For each subinterval $[t_1,t_2] \subseteq  [0,1]$ whose image under $\gamma$ lies in a neighborhood $U \in \mathcal{U}$, we claim that compatibility with all simplified exit paths beginning on $\gamma([t_1,t_2])$ follows from compatibility with simplified exit paths beginning at any fixed point of $\gamma([t_1,t_2])$.  For any $t_3 \in [t_1,t_2]$ and any simplified exit path $\beta$ from $\gamma(t_3)$ to a point $x \in U \cap S_j$ for $j \neq i$, we can form a simplified exit path $\alpha$ from $\gamma(t_1)$ to a point $y \in U \cap S_j$ in the same stratum as $x$ and a path $\gamma'$ from $x$ to $y$ in $U \cap S_j$. Since $\cF|_{U \cap S_i}$ is constant, the two composite paths from $\gamma(t_3)$ to $y \in U$ must have parallel transport yielding the unique extension of $f_{\gamma(t_3)}$ to a section of $U$.  Since we already have sections defined on all strata, it follows that compatibility with $\beta$ is equivalent to compatibility with $\alpha$.  

    Since $[0,1]$ is covered by subintervals $[t_1, t_2]$ as in the preceding paragraph, we obtain that compatibility with exit paths departing $\gamma(0)$ is equivalent to compatibility at $\gamma(1)$.   
\end{proof}

In the situation at hand, given a simplified exit path $[0,1] \to X$, the parallel transport of $s \in \cF_{\gamma(0)}$ along $[0,\varepsilon]$ is given by restricting a local partial resolution to neighborhoods of nearby points  to $\gamma(0)$, followed by parallel transport in the stratum.   The same argument works when including also the datum of a relatively ample line bundle.

We note that it is a consequence of the constructibility of the sheaves and \cite[Theorem 1.2]{TrEP} that the above actually defines functors $\Ex(X,\cS) \to \bf{Sets}$ such that sections are given by stalks compatible by exit paths, although we do not need to use this.  (This statement, in addition to the above, says that the above parallel transport for exit paths does not depend on any choices and is independent of tame homotopy of exit paths.)

\subsection{Compatibility across strata} \label{ss:compatibility}
For a stratified variety $(X, \cS)$, the set of strata has a partial order given by $S \leq S'$ if $S \subseteq \overline{S'}$. We recall a few elementary notions for partially ordered sets. 

\begin{defn} 
The \emph{Hasse diagram} for a partially ordered set $(X, \leq )$ is a directed\footnote{Note that most authors draw the Hasse diagram as an \emph{undirected} graph, but place the vertex for $x$ below the vertex for $x'$ if $x < x'$ to indicate direction.} graph with vertex set $X$ and an arrow $x \rightarrow x'$ if $x < x'$ and no $x''$ satisfies $x < x'' < x$. 
\end{defn}

\begin{defn} 
Fix a partially ordered set $(X, \leq)$ and two elements $x_1, x_2 \in X$. When they exist, the \emph{join} $x_1 \vee x_2$ is the supremum (i.e., least upper bound) of $x_1$ and $x_2$, and the \emph{meet} $x_1 \wedge x_2$ is the infimum (i.e., greatest lower bound) of $x_1$ and $x_2$. 
\end{defn}

Thanks to our procedure, we can build resolutions up from the minimal strata.  Two distinct minimal strata have no meet, and the compatibility condition can be checked in their join. If the join is an open stratum, the compatibility condition is automatically satisfied since that stratum has a unique (partial) crepant resolution, namely the identity.  Similarly, the compatibility condition is satisfied for smooth crepant resolutions if the join has codimension at most two, because there there is a unique crepant resolution (the minimal one).

For constructing resolutions, in the absence of monodromy, one can always extend a crepant resolution up a chain $S_1 \leq S_2 \leq \cdots \leq S_n$. In the following special case involving products, we can construct crepant resolutions \emph{down} the Hasse diagram (with no obstructions):

\begin{equation} \label{v-diagram} \xymatrix{  S_1 && S_2  & \\  
& S_3  \ar[lu] \ar[ru] & 
} \end{equation}\\
\hspace*{3.2cm} \text{where the slice to $S_3$ is a product of slices to $S_1$ and $S_2$.}\\

Here, we mean that for $x \in S_3$, there is a local neighborhood of $x$ of the form $U_3 \times F_1 \times F_2$ where $U_3$ is a neighborhood of $x$ in $S_3$, and $U_3 \times F_1$ and $U_3 \times F_2$ map isomorphically to neighborhoods of $x$ in $\overline{S_2}$ and $\overline{S_1}$, respectively (so the slice to $x$ in $S_1$ is $F_1$ and the slice to $x$ in $S_2$ is $F_2$). The proof of this relies on the following relative, analytic analogue of \cite[Lemma 5.2]{BCS21}:

\begin{lem} \label{l:Pic_product}
    Let $\rho_i: \tilde{U}_{i} \rightarrow U_i$ be smooth projective crepant resolutions with $U_i$ contractible Stein manifolds, for $i = 1, 2$. Then there is an isomorphism of analytic relative Picard groups: 
    \[
    \Pic(\tilde{U}_1 \times \tilde{U}_2 / U_1 \times U_2) \cong \Pic(\tilde{U}_1/ U_1) \times \Pic(\tilde{U}_2 / U_2).
    \]
    Moreover, under this isomorphism, $\Mov(\tilde U_1 \times \tilde U_2/U_1 \times U_2) = \Mov(\tilde U_1/U_1) \times \Mov(\tilde U_2/U_2)$, with Mori equivalence relation taken to the product of the Mori equivalence relations.
\end{lem}

\begin{proof}
First, as explained in Remark \ref{r:absolute-relative}, in this case the absolute Picard groups equal the relative ones.  By Grauert--Riemenschneider vanishing, $R\rho_* \cO_{\tilde U_i} \cong \cO_{U_i}$. Since the $U_i$ are Stein, this implies that $H^{>0}(\tilde U_i, \cO_{\tilde U_i})=0$. The same is true for the product $\tilde U_1 \times \tilde U_2$. By the exponential sequence we therefore obtain $\Pic(\tilde U_1 \times \tilde U_2) \cong H^2(\tilde U_1 \times \tilde U_2, \bZ)$, and $\Pic(U_i) \cong H^2(\tilde U_i, \bZ)$. 

Next, for any crepant resolution $\tilde U \to U$ with $U$ contractible Stein, we claim that $H^1(\tilde U, \bZ) = 0$ (cf.~\cite[Corollary 1.5]{KaledinSurvey}). To see this, we apply the exponential sequences for $\tilde U$ and $U$:
\[
\xymatrix{
H^0(\tilde U, \cO_{\tilde U}) \ar@{=}[d] \ar[r] & H^0(\tilde U, \cO^\times_{\tilde U}) \ar[r] \ar@{=}[d] & H^1(\tilde U, \bZ) \ar[d] \ar[r] & H^1(\tilde U, \cO_{\tilde U}) \ar@{=}[d] \\
\Gamma(U, \cO_U) \ar[r] & \Gamma(U, \cO_U^\times) \ar[r] & 0 \ar[r] & 0.}
\]
with the first two equalities coming from properness of $\rho$.  We conclude that
\[H^1(\tilde U, \bZ) \cong \operatorname{coker}(H^0(\tilde U, \cO_{\tilde U}) \to H^0(\tilde U,\cO_{\tilde U}^\times)) \cong \operatorname{coker}(\Gamma(U,\cO_U) \to \Gamma(U, \cO_U^\times)) = 0.\]

By the K\"unneth theorem and connectedness of $\tilde U_i$  we therefore have $H^2(\tilde U_1 \times \tilde U_2) \cong H^2(\tilde U_1) \oplus H^2(\tilde U_2)$, which by the first paragraph establishes the first assertion. 

For the second assertion, let $\pi_1,\pi_2: \tilde U_1 \times \tilde U_2 \to \tilde U_i$. We note that, for line bundles $L_i$ on $\tilde U_i$, the base locus $B(\pi_1^* L_1  \otimes \pi_2^* L_2)$ is the union of the pullbacks of base loci, $\pi_1^{-1}B(L_1)$ and $\pi_2^{-1} B(L_2)$. Finally, for $L_i \in \Mov(\tilde U_i/U_i)$, we have, for modifications (see Definition \ref{d: mod-mori}):
\[
\tilde U_1 \times \tilde U_2 (\pi_1^* L_1 \otimes \pi_2^* L_2) \cong \tilde U_1(L_1) \times \tilde U_2(L_2).
\]
As a result, the Mori equivalence relation for $\tilde U_1 \times \tilde U_2$ is the product of the Mori equivalence relations for $\tilde U_1$ and $\tilde U_2$. \qedhere
\end{proof}
\begin{rem}
More generally, we can allow $\tilde U_i$ to be singular and we don't even need $\rho_i: \tilde U_i \to U_i$ to be crepant, as long as $U_i$ and $\tilde U_i$ have rational singularities (e.g., canonical singularities by \cite[Theorem 5.22]{KollarMori}),
and $\rho_i$ is locally projective and birational.  Under the latter hypotheses, $R (\rho_i)_* \mathcal{O}_{\tilde U_i} \cong \cO_{U_i}$ by \cite[Theorem 1.4]{Kovacs-rational}, so the result of Lemma \ref{l:Pic_product} and the proof holds. In particular, this is true for any $X$ satisfying (*) and any $\rho_i: \tilde U_i \to U_i$ in $\ccR(U_i)$.
\end{rem}

As a consequence of the lemma (since we have a basis of contractible Stein neighborhoods, as noted in Remark \ref{r:absolute-relative}), a local projective crepant resolution at a point of the stratum $S_3$ is uniquely determined by local projective crepant resolutions at nearby points of  $S_1$ and  $S_2$. Therefore, we need not consider such strata when classifying crepant resolutions.

\begin{rem}
Under a mild condition, a diagram as in \eqref{v-diagram} with the product condition for slices exists whenever the strata $S_3$ has decomposable slice. Namely, suppose that all the slices to the strata have exceptional basepoint (meaning that, in a neighborhood of the basepoint, all vector fields on the slice vanish at the basepoint). This includes the case that the 
stratification is by iterated singular loci (since by locally integrating vector fields, singular loci are preserved). In particular, if $X$ is Poisson with a (locally finite) stratification by symplectic leaves, then the stratification is also the one by iterated singular loci, and the condition holds.

Then such a diagram occurs whenever a stratum $S_3$ exists with slice decomposing as a product $F_1 \times F_2$: Simply let $S_1$ be the generic stratum which near $s_3 \in S_3$ lies in  $U_3 \times \{0\} \times F_2$, and similarly define $S_2$ as the
generic stratum which near $s_3$ lies in $U_3 \times F_1 \times \{0\}$. 
 
Without the assumption that basepoints of slices are exceptional, we can still define for a stratum $S_3$ with decomposable slice $F_1 \times F_2$ (with $F_1, F_2$ both singular) strata $S_1, S_2$ such that crepant resolutions for $S_i$ uniquely determine one for $S_3$. Indeed, 
by locally integrating  vector fields, after shrinking the neighborhood we can further decompose $F_i$ as $F_i' \times D_i$ for $D_i$ some disc, and the basepoint of $F_i'$ is exceptional.  In this case, there will have to be strata $S_i$ with slices $F_i' \times  D_1 \times D_2$. In this case, though $S_3$ will not be  a product of the slices to $S_1$ and $S_2$ (unless the $D_i$ are zero-dimensional), it is still true that crepant resolutions for $S_1$ and $S_2$ uniquely determine that of $S_3$, since local crepant resolutions of a variety and of a variety times a disc are in bijection. 
\end{rem}

Summarizing this subsection:
\begin{itemize}
    \item[(1)] a choice of local projective crepant resolutions at the minimal strata can determine at most one locally projective crepant resolution of $X$,
    \item[(2)] we need not check compatibility for strata whose slices are products of two singular varieties (provided we determine their local crepant resolutions from factor slices), and
    \item[(3)] compatibility is automatic across strata whose join is an open stratum, or in the case of full crepant resolutions, for strata whose join is in codimension two.
\end{itemize}

\subsection{Corollaries of the main theorem}
Let $(X, \cS)$ be as in Theorem \ref{t:constr}. We will record some consequences of Theorem \ref{t:constr}. For the first set of corollaries, we return to the setting of smooth projective crepant resolutions, as in the introduction. 

\begin{cor} \label{c: unique locally-proj resolution}
Let $\{ S_i \}_{i \in I}$ denote the \emph{closed} strata in $\cS$ and pick basepoints $s_i \in S_i$. If the germ of $X$ at $s_i$ has a unique local projective crepant resolution for all $i$, then they glue to form a global locally projective crepant resolution.

More generally, if this is true for all strata except for isolated singularities, then locally projective crepant resolutions of $X$ are in bijection with choices of local projective crepant resolutions of the isolated singularities. 
\end{cor}

\begin{proof} 
It suffices to prove the assertion of the second paragraph. We use the subsheaf of $\ccR$ of smooth crepant resolutions (not the partial ones).  All exit paths necessarily take the unique values everywhere except the isolated singularities, which can therefore take any value to get a global section. The result then follows from Theorem \ref{t:constr}.
\end{proof}

\begin{cor}
Let $X$ be a variety whose singularities are all either isolated, or have du Val singularities. 
Then locally projective crepant resolutions are in bijection with arbitrary choices of local projective crepant resolutions of the isolated singularities (when they exist). 
\end{cor}

Here, by ``having du Val singularities'' we mean that some neighborhood is isomorphic to the product of a disc of some dimension with a du Val singularity. (Since du Val singularities have no moduli, this is equivalent to being a codimension two singularity whose generic two-dimensional slice is du Val, see \cite[Corollary 1.14]{Reid-c3f}.)

\begin{proof}
    Let $Z \subseteq X$ be the singular locus.  This is a Zariski subset which, by assumption, has only isolated singularities.  At any smooth point, by assumption it is analytically locally the product of a du Val singularity and a smooth disc. So, $X$ is stratified, with three types of strata: the open strata, the strata which locally have du Val singularities, and the isolated strata. 
     Since du Val singularities all have unique projective crepant resolutions, the result now follows from Corollary \ref{c: unique locally-proj resolution}.
\end{proof}

\begin{cor}
Let $X$ be either a $4$-dimensional variety with symplectic singularities, or a $3$-dimensional canonical singularity. Then locally projective crepant resolutions of $X$ are in bijection with arbitrary choices of local projective crepant resolutions of the isolated singularities (when they exist). 
\end{cor}

\begin{proof}
Note that two-dimensional canonical singularities, two-dimensional symplectic singularities, and du Val singularities all coincide. Also, canonical and symplectic singularities are normal by definition, hence smooth in codimension one.  Therefore, it suffices to show that canonical $3$-dimensional and symplectic $4$-dimensional singularities are stratified, with all non-isolated, non-open strata of codimension two.  For symplectic singularities, this follows from the stratification into symplectic leaves.  For canonical three-dimensional singularities, this was explained in \cite[Corollary 1.14]{Reid-c3f} (as cited already).
\end{proof} 

\begin{cor} \label{c: subset with same stratification}
Let $(X, \cS)$ be a stratified variety satisfying Definition \ref{d:local product} and $\cF$ a constructible sheaf.
Let $U \subset X$ be an open subset of $X$ such that for each stratum $S_i \in \cS$, $U \cap S_i$ is connected and nonempty and the inclusion map $\iota: U \cap S_i \hookrightarrow S_i$ induces a surjection $\pi_1(U \cap S_i) \twoheadrightarrow \pi_1(S_i)$ on fundamental groups. Then the restriction map induces a bijection 
$\Gamma(X,\cF) \to \Gamma(U,\cF)$.
\end{cor}
Note that, under our assumptions, the subvariety $U$ inherits a 
stratification by locally closed connected strata $U \cap S_i$ satisfying Definition \ref{d:local product}.

\begin{proof}
Sections on $U$ and on $X$
can both be identified with choices of sections at every point compatible under the exit paths listed in Proposition \ref{p:compatibility-limited-exitpaths}.  If we choose basepoints of $X$ to be in $U$, then the list of exit paths can be chosen to be exactly the same (for item (2), it's because of the assumption on surjectivity of $\pi_1(S_i \cap U) \to \pi_1(S_i)$, together with the bijection between the connected strata, and for item (3), it's because an open neighborhood of a point of $s_i$ is contained in both $U$ and $X$, along with the bijection of connected strata), 
except we additionally need for sections on $X$  to have compatibility with a single path from the basepoint of each stratum $S_i$ to each point of $S_i \setminus U$.  The latter just says that the sections at $X \setminus U$ are uniquely determined from those on $U$.
\end{proof}

\begin{cor} \label{cor:unique_minimal}
    Let $(X, \cS)$ be a stratified variety with a unique minimal stratum given by a point $S_{\min} = \{s \}$, satisfying Definition \ref{d:local product}. Then there is an injection $\ccR(X) \hookrightarrow \ccR(U_s)$ for $U_s$ any neighborhood of $s$, and the same holds replacing $\ccR$ with $\widetilde{\ccR}$ and $\cP$. If there exists a neighborhood $U_s$ isomorphic to the neighborhood of $0$ in a conical variety $Y$, then there is an injection $\ccR(X) \hookrightarrow \ccR(Y)$. 
\end{cor}
In particular, if $Y= \cM_{0, 0}(Q, d)$  
    is a (conical) Nakajima quiver variety (see Subsection \ref{ss:MQV} for the definition) for an unframed quiver $Q$ with some $d_i = 1$ such that there exists a simple representation,
    then the set $\ccR(\cM_{0, 0}(Q, d))$ into which $\ccR(X)$ injects can be described by GIT chambers following \cite[Corollary 4.7]{BCS22}.
\begin{proof}
The first statement is an immediate consequence of the constructibility of the sheaves (or of the exit path description of global sections).
For a conical variety $X$ with cone point 0, one can choose a stratification $\cS$ where each stratum $S_i$ is conical (i.e., $0 \in \overline{S}_i$). Therefore, compatibility can be checked in any neighborhood of 0. 

Next notice that if a local resolution has monodromy along the loop $\gamma$ in a stratum $S$, then the non-triviality of the homotopy class $[\gamma]$ implies $\gamma$ has a non-zero winding number about 0. Consequently the resolution does not arise as an extension of one in a neighborhood of the cone point. 

Therefore, any resolution of any neighborhood of 0 extends uniquely to $X$. 
\end{proof}

\subsection{Reduction to combinatorics and linear algebra} \label{ss: combinatorics}

Our main result says the assignment of $U \subset X$ open to the set $\ccR(U)$ of isomorphism classes of projective crepant resolutions on $U$ is an $\cS$-constructible sheaf. Such an $\cS$-constructible sheaf can in turn be described as a functor $F: \Ex(X, \cS) \rightarrow \bf{Sets}$. By choosing the auxiliary data of a basepoint $s \in S$ for each stratum in $\cS$ and a generating set of simplified exit paths $\{ \gamma_{s, s'} \}$, such a functor is a system of sets $F(s) := \ccR_{s}$ with generalization maps $\gen_{s, s'} := F(\gamma_{s, s'}): \ccR_{s} \rightarrow \ccR_{s'}$ if $\overline{S'} \cap S \neq \emptyset$.

In nice cases, for each pair of strata, it suffices to choose only one  exit path between them (as a consequence, we only need to do this for neighboring strata), as we now explain.  Suppose $s_i \in S_i$ for $i=1, 2$ and $s_2 \in \overline{S}_1$. Suppose $s_2$ has a neighborhood basis $\{ U_i \}$ such that $U_i \cap S_1$
is connected (hence path-connected, since $S_1$ is a manifold). Then any two simplified exit paths from $s_2$ into $S_1$ differ by an invertible morphism in the exit path category. A stronger, sufficient condition that ensures the need for only a single exit path from $s_2$ into $S_1$ is that $\overline{S}_1$ is \emph{topologically unibranch} at $s_2$. 

\begin{defn} \cite[Definition (3.9)]{Mumford81}
Let $X$ be a variety
and fix $x \in X$. We say $X$ is \emph{topologically unibranch} at $x$ if for all Zariski closed subsets $Y \subset X$, $x$ has a neighborhood basis $\{ U_n \}$ in the complex topology such that $U_n \backslash (U_n \cap Y)$ is connected. 
\end{defn}
This property is equivalent to the statement that the fiber of $x$ in the the normalization of $X$ is a single point (see the exposition in \cite[Section 5.6]{MT15} on Zariski's main theorem \cite{Zariski43}, where this is the claim U3 $\iff$ U4).

Note that, by local triviality along a stratum and the assumption that strata are connected, these properties hold at  $s_i$
if and only if they hold at any other point of the same stratum.  In this situation, we can refine Proposition \ref{p:compatibility-limited-exitpaths} to only need a single simplified exit path at each $s_i$ to each stratum $S_j$ with $\overline{S_j} \ni s_i$. As we observed already, these are the same as the strata with
$\overline{S_j} \supseteq S_i$.  Finally,  by composing these with an arbitrary path from the endpoint of the simplified exit path to the basepoint $s_j$, we obtain the following:
\begin{cor} Consider a stratified variety satisfying Definition \ref{d:local product}. Suppose that for every pair of strata $S_i, S_j$ with $\overline{S_j} \supseteq S_i$, $\overline{S_j}$ is topologically unibranch at $s_i$ (or just that there is a neighborhood basis $\{U_k\}$ of $\overline{S_j}$ at $s_i$ such that $U_k \cap S_j$ is connected for all $k$).  Then global sections of a constructible sheaf $\cF$ are in bijection with choices of sections at basepoints $s_i$ which are compatible with parallel transport along:
\begin{enumerate}
\item[(1)] Any set of closed paths generating the fundamental groups $\pi_1(S_i, s_i)$;
\item[(2)] For each pair of distinct strata $S_i, S_j$ with $\overline{S_j} \supset S_i$, any single exit path from $s_i$ to $s_j$.
\end{enumerate}
\end{cor}
Note that in (2) in the corollary, we can take the path to be a simplified exit path.  Moreover, since compositions of exit paths are exit paths, we can also restrict in (2) to pairs of \emph{neighboring} strata, i.e., strata such that there does not exist an intermediate stratum $S_k \notin \{S_i, S_j\}$ with $\overline{S_j} \supset S_k, \overline{S_k} \supset S_i$.  These are the strata that are adjacent in the poset of strata under inclusion of closures, (i.e., adjacent in the Hasse diagram).

In view of 
Theorem \ref{t:P-constr}, for each of the paths above, we only need to compute the parallel transport as a linear map on the sheaf $\cP^\sharp$ of local relative Picard spaces. Then, the global locally projective resolutions are given by choices of Mori cones at each basepoint compatible under this transport, and the global projective resolutions with fixed relatively ample bundle are given by choices of movable line bundles compatible under this transport. In this sense, we have reduced the classification of crepant resolutions to combinatorics and linear algebra, provided we have computed the local classification and the parallel transport.


\begin{rem}
The unibranch condition is not always satisfied in examples.  For example, if $X$ is the nilpotent cone of a semisimple Lie algebra, stratified by the adjoint orbits, it can happen that an orbit closure is not unibranch. This is true even though $X$ admits a (unique) projective crepant resolution, the Springer resolution. 

For instance, let $S_1$ be the nilpotent orbit in $\mathfrak{so}_7$ labeled by the partition $(3,2,2)$, of dimension $12$. The orbit closure $\overline{S_1}$ is not unibranch at a point of the orbit $S_2$ labeled by $(2,1^4)$ of dimension $10$ (see \cite[p.~595]{KP-norm-cl}). 
\end{rem}




\section{Examples} \label{s: examples}

\subsection{Monodromy and incompatibility}
We now construct examples with monodromy obstructions and examples with compatibility obstructions to gluing local crepant resolutions.

The below examples are either Nakajima quiver varieties or nilpotent orbit closures in $\mathfrak{sl}_n(\bC)$. For the first type it is well known that the slice to the strata are also Nakajima quiver varieties. In the case where variation of stability produces a smooth symplectic resolution, it also produces one for each of the slices, so that condition (*) is satisfied.  For the second type, the Springer resolution produces resolutions for the Kostant--Slodowy slices to each of the strata, so that again (*) is satisfied.  (It is also true that, since these examples are algebraic and have canonical singularities, (*) is a general consequence of \cite{BCHM10}, see Remark \ref{r:bchm}.)

\begin{exam} \label{ex: monodromy} (Monodromy can occur)
Let $Q$ be the quiver with a single vertex and $g \geq 3$ loops. Pick the dimension vector $\alpha = (2)$. By Bellamy--Schedler \cite{BS21}, the quiver variety (see Subsection \ref{ss:MQV})
\[
\mathcal{M}_{0, 0}(Q, (2)) = \left \{ (X_1, Y_1, \dots,  X_g, Y_g) \in  \End_{\bC}(\bC^2) \ \middle | \ \sum_{i=1}^g [X_i, Y_i] =0  \right \} /\!/ GL_2(\bC)
\]
is Zariski locally factorial, terminal, and singular and hence does not admit an algebraic symplectic resolution. The same holds for the variety $X := \mathcal{M}_{0, 0}(Q, (2)) \backslash \{ 0 \}$, which we claim has an analytically local symplectic resolution around every point. To see this notice that the stratification of $\mathcal{M}_{0, 0}(Q, (2))$ into its symplectic leaves is given on semisimple representations (i.e., elements of closed $\GL_2$-orbits in $\mu^{-1}(0)$) by:
\[
    S_0 := \{ \rho = \rho_1 \oplus \rho_2 \ | \ \rho_1 \cong \rho_2 \} \hspace{1cm} 
    S_1 := \{ \rho = \rho_1 \oplus \rho_2 \ | \  \rho_1 \not \cong \rho_2 \} 
    \hspace{1cm}
    S_2 := \{ \rho \text{ simple} \}.
\]
This variety has dimension $2-4(2-2g)=8g-6$.
The minimal stratum $S_1$ of $X$ has local quiver
\[
\begin{tikzcd}[ampersand replacement=\&]  \arrow[loop left, distance=2em, start anchor={[yshift=-1ex]west}, end anchor={[yshift=1ex]west}, "g"]  \bullet  \arrow[r, "2g-2"] \& \bullet \arrow[loop right, distance=2em, start anchor={[yshift=1ex]east}, end anchor={[yshift=-1ex]east}, "g"] \end{tikzcd} 
\]
 with $4g-2$ total arrows and dimension vector $(1, 1)$. This is not smooth: its singular locus is the minimal leaf  corresponding to all the arrows of the doubled quiver between two distinct vertices being zero. This singular locus has dimension $4g$, corresponding to the loops of the doubled quiver being arbitrary scalars.   It admits a symplectic resolution since the dimension vector $(1, 1)$ is indivisible, i.e., the greatest common divisor of its entries is one.  As the singular locus has codimension $4g-6 > 2$, and hence the exceptional locus has codimension greater than one, the symplectic resolutions are small (contract no divisors). By \cite[Theorem 1.2]{BCS22}, there are two non-isomorphic projective symplectic resolutions, given by the two inequivalent nonzero stability conditions $(1,-1)$ and $(-1,1)$.  The complement of this leaf is the smooth open leaf of simple representations.
Hence $X$ admits an analytically local symplectic resolution in some neighborhood of every point, yet these do not glue to a global algebraic symplectic resolution. Furthermore, there is a monodromy obstruction to an analytic locally projective crepant resolution: a path in the variety which swaps the isomorphism classes of $\rho_1$ and $\rho_2$ will swap the two stability conditions, thus interchange the two inequivalent resolutions.
\end{exam}

\begin{exam} \label{exam: incompatibility} (Incompatibility can occur)
Let $\overline{\cO}$ be the minimal nilpotent orbit closure of $\fsl_3(\bC)$, defined explicitly as $\{ A \in \fsl_3(\bC) : A^2 = 0 \}$. Note that $\overline{\cO}$ has a stratification into $\{ 0 \}$ and $\cO$. Neither stratum can have monodromy since $\pi_1(\{ 0 \} , 0) =1$ and $\ccR_s$ is a singleton set for $s \neq 0$. (In fact, it is also true that $\pi_1 \cO = 1$, since the partial Springer resolution provides an isomorphism $\cO \cong (T^* \bP^2) \setminus \bP^2$, but we do not need this.)
Further, $\overline{\cO}$ admits two symplectic resolutions of the form $T^*(\SL_3(\bC)/P) \rightarrow \overline{\cO}$ corresponding to the two parabolic subgroups of $\SL_3(\bC)$. We can also write these resolutions as $T^*\text{Gr}(1,3)$ and $T^*\text{Gr}(2,3)$, or the cotangent bundle of $\bP^2$ and its dual.  So the product $\overline{\cO} \times \overline{\cO} \times \overline{\cO}$ has strata with no monodromy and admits $2^3$ symplectic resolutions.

 Let $Y = \overline{\cO} \times \overline{\cO} \times \overline{\cO} \backslash \{ (0, 0, 0) \}$. There are three minimal leaves: $\overline{\cO} \times 0 \times 0$,  $0 \times \overline{\cO} \times 0$, and  $0 \times 0 \times \overline{\cO}$. Each has slice isomorphic to $\cO \times \cO$ and hence has $4$ symplectic resolutions. Of the $4^3 = 64$ possible choices, only 8 can result in an actual symplectic resolution. Indeed, the choices can only be compatible when they come from a triple of choices by restriction, i.e., from a choice of resolution at the cone point we threw out. This can be seen by restricting from each pair of minimal leaves to their product. 
\end{exam}

\begin{exam} \label{exam:artificial} (Constructing monodromy)
   Let $X$ be a symplectic singularity with two symplectic resolutions $X_1$ and $X_2$ that differ by composing with an automorphism $\phi: X \to X$ of the singularity (in particular, $X_1$ is isomorphic to $X_2$ as varieties).
   Let $S$ be symplectic with an automorphism $\psi$ whose non-identity powers $\{ \psi^{n} \hspace{.01cm} \mid \hspace{.01cm} \psi^{n} \neq \text{id}_{S} \}_{n \in \bN}$ have no fixed points. Assume further that the order of $\phi$ divides the order of $\psi$ (or the order of $\psi$ is infinity). The quotient  $(X \times S) / G$ by the group $G := \langle \phi \times \psi \rangle$ has non-trivial monodromy along a path from $(x, s)$ to $(\phi(x), \psi(s))$ in $X \times S$.

   For example, for $n, m \in \bN$ with $m < n/2$ define 
   \[
   X_{n, m} = \{ A \in \fsl_n(\bC) \ \mid \ A^2 = 0, \text{ rank}_{\bC}(A) \leq m \}.
   \]
   There are non-isomorphic Springer resolutions $T^*(\text{Gr}(m, n)), T^*(\text{Gr}(n-m, n) )\rightarrow X_{n, m}$ whose sources are isomorphic as varieties (via an inner product $\langle - , - \rangle$ on $\bC^n$ that identifies each $m$-plane with its orthogonal complement $(n-m)$-plane), and this isomorphism induces a nontrivial automorphism of $X_{n,m}$, namely the transpose. 
   Note that $\phi$ squares to the identity. Let $S = (\bC)^2 \backslash \{ (0, 0) \}$ with action of $C_2 = \langle \psi \rangle$ taking $(x, y) \mapsto (-x, -y)$. Then $(X_{n,m} \times S) / ( (A, x, y) \sim (A^{\vee}, -x, -y))$ has monodromy taking $T^*(\text{Gr}(m, n))$ to $T^*(\text{Gr}(n-m, n))$ along a path from $(A, x, y)$ to $(A^{\vee}, -x, -y)$ in $X \times S$.
\end{exam}

\subsection{Symmetric powers of surfaces with du Val singularities} \label{ss: symmetric_powers}

Let $Y$ be a surface with only du Val singularities and a symplectic form on the smooth locus. Consider $X := \Sym^n(Y) := Y^{\times n} / \fS_n$ where $\fS_n$ denotes the symmetric group on $n$-letters acting by permuting the factors. 
Then $X$ has a stratification by symplectic leaves (equivalently, the singularity stratification):
\[
X = \bigsqcup_{
\resizebox{0.1\textwidth}{!}{
$\begin{array}{c} f: Y^{\sing} \rightarrow \bN \\ \sum f(z) \leq n \end{array}$
}
} X_{f} \hspace{1cm} 
X_f := \bigsqcup_{
\resizebox{0.1\textwidth}{!}{
$\begin{array}{c} f: Y^{\sing} \rightarrow \bN \\ \sum_z f(z) \leq n \end{array}$
}
} 
\Sym^{n- \sum_z f(z)}(X \backslash X^{\sing} ) \times \prod_{z \in Y^{\sing}} (z, \dots, z) \ ( f(z) \text{ times}).
\]

Note that the singularities satisfy (*) since the slice to the strata are products of symmetric powers of du Val singularities (or $\bC^2$), and each of these can be resolved  by the composition $\Hilb^{m}(\widetilde{\bC^2/\Gamma}) \to \Sym^m(\widetilde{\bC^2/\Gamma}) \to \Sym^m(\bC^2/\Gamma)$ of the Hilbert--Chow morphism and the symmetric power of the minimal resolution of the surface.  (Also, analytically locally the variety is isomorphic to complex algebraic varieties, so  Remark \ref{r:bchm} applies to this case.)
 
\begin{prop}\label{p:sym-duval}
Every locally projective partial crepant resolution of $X$ is uniquely determined by its restriction to open neighborhoods of the points $(z, \dots, z)$ ($n$ times) for $z \in Y^{\sing}$. Consequently, for $U_z$ disjoint neighborhoods of the singularities $z$ of $Y$,
  \[
  \ccR(Y) \cong \prod_{z \in Y^{\sing}} \ccR(\Sym^{n}( U_z ))  
  \cong \prod_{z \in Y^{\sing}} \ccR(\Sym^{n}( \widehat{\bC^2/ \Gamma_z})) 
  \]
  for some $\Gamma_z \subset \text{SL}_2(\bC)$. 
  Hence each locally projective partial crepant resolution of $Y$ is locally given by variation of stability parameter. 
  
  Moreover, if $Y$ has finitely many singularities, then every locally projective partial crepant resolution of $X$ is globally projective. 
\end{prop}

By \cite[Corollary 1.3]{BC20}, $\ccR(\Sym^{n}( \widehat{\bC^2/ \Gamma_z}))$ can be identified with chambers in the GIT fan modulo the Weyl group action. The later can be computed explicitly, for $W$ with Coxeter number $h$ and exponents $e_1, \dots, e_{\ell}$
\[
\# \ccR(\Sym^{n}( \widehat{\bC^2/ \Gamma_z}))   = \prod_{i=1}^{\ell} \left ( \frac{(n-1)h}{e_i +1} + 1\right )
\]
see \cite[Proposition 1.2]{Bellamy16}. For example, in type $E$, writing $T, O, I$ for the binary tetrahedral, binary octahedral, and binary icosahedral groups respectively, we have
\begin{align*}
\# \ccR(\Sym^{n}( \widehat{\bC^2/ T}))  &= \frac{n}{30} ( 1728 n^5 - 4320 n^4 + 4140 n^3 - 1900 n^2 + 417 n - 35) \\
\# \ccR(\Sym^{n}( \widehat{\bC^2/ O}))  &= \frac{n}{280} ( 59049 n^6 - 183708 n^5 + 229635 n^4 - 147420 n^3 + 51156 n^2 - 9072 n + 640) \\
\# \ccR(\Sym^{n}( \widehat{\bC^2/ I}))  &= \frac{n}{1344} ( 1265625 n^7 - 4725000 n^6 + 7323750 n^5 - 6100500 n^4 + 2943325 n^3 \\
& \hspace{2cm} - 820260 n^2 + 121796 n - 7392).
\end{align*}

\begin{proof}
Denote the singularities of $Y$ by $z_i$, i.e., $Y^{\sing}=\{z_i\}$. For each $z_i \in Y$ pick $U_i \subseteq Y$ a contractible neighborhood, such that the collection $\{ U_i \}$ is pairwise disjoint. 

Let $V = Y \backslash \left ( \sqcup_i \overline{U_i} \right )$ where $\overline{U_i}$ denotes the closure of $U_i$ in $S$. The inclusion map $\iota_V: V \rightarrow Y$ induces a surjection on the fundamental group of the smooth locus $Y^{\sm}$ and the inclusion $\iota_U: \left ( \sqcup_i U_i \right ) \rightarrow Y$ induce surjections on the fundamental groups of the singular strata, $Y^{\sing} = \{ z_i \}$. So their union $\iota: V \sqcup \left ( \sqcup_i U_i \right ) \rightarrow X$ induces surjections on the fundamental group of each stratum. Therefore, the induced inclusion on the $n$th symmetric power $\Sym^n(\iota): \Sym^n(V \sqcup \left ( \sqcup_i U_i \right )) \rightarrow \Sym^n(Y)$ induces surjections on the fundamental groups of each stratum. For each stratum in $\Sym^n(Y)$, there is a stratum in $\Sym^n(V \sqcup \left ( \sqcup_i U_i \right ))$
mapping into it which is of the form 
\[S \times \prod_i (z_i,\ldots,z_i) \subseteq \Sym^m V \times \prod_i \Sym^{m_i} U_i \subseteq \Sym^n(V \sqcup \left ( \sqcup_i U_i \right ) ).
\]
By Corollary \ref{c: subset with same stratification}, we can detect monodromy of a stratum of $\Sym^n(Y)$ by monodromy on the source stratum. This is a product of $S$ with a point, so the monodromy is detected by the monodromy along $S \subseteq \Sym^m V$. Since there is locally a unique crepant resolution along this stratum, there can be no monodromy.
Consequently, there is no obstruction extending local crepant resolutions to entire stratum. 

Let $z_1$ and $z_2$ be du Val singularities in the surface $Y$. Fix $0 <m < n$. Define the strata in $X = \Sym^n(Y)$
\[
S_{a, b, c} := \{ (\underbrace{z_1, \dots, z_1}_\text{a copies}, y_1, y_2, \dots, y_b, \underbrace{z_2, \dots, z_2}_\text{c copies}) \mid y_i \in Y^{\sm}, y_i \neq y_j \text{ for } i \neq j \} \hspace{1cm} a + b + c = n.
\]
The four chains: 
\[
S_{n, 0, 0} \rightarrow S_{n-1, 1, 0} \rightarrow \cdots \rightarrow S_{m, n-m, 0}
\hspace{1cm}
S_{m, 0, n-m} \rightarrow S_{m, 1, n-m-1} \rightarrow \cdots \rightarrow S_{m, n-m, 0} 
\]
\[
S_{0, 0, n} \rightarrow S_{0, 1, n} \rightarrow \cdots \rightarrow S_{0, m, n-m}
\hspace{1cm}
S_{m, 0, n-m} \rightarrow S_{m-1, 1, n-m} \rightarrow \cdots \rightarrow S_{0, m, n-m}
\]
form an M-configuration in the Hasse diagram
\[
\xymatrix@C=-30pt{
 S_{m, n-m, 0} = \{ (\overbrace{x_1, \dots, x_1}^\text{m copies}, y_1, \dots, y_{n-m}) \} && S_{0, m, n-m} = \{ (y_1, \dots, y_{m}, \overbrace{x_2, \dots, x_2}^\text{n-m copies}) \} & \\
S_{n, 0, 0} = \{ (x_1, \dots, x_1) \} \hspace{1cm} \ar[u] & S_{m, 0, n-m} = \{ (\underbrace{x_1, \dots, x_1}_\text{m copies}, \underbrace{x_2, \dots, x_2}_\text{n-m copies}) \} \ar[ru] \ar[lu] & \hspace{1cm} S_{0, 0, n} := \{ (x_2, \dots, x_2) \} \ar[u].
}
\]
The slice to $S_{m,0,n-m}$ is a product of slices to $S_{m,n-m,0}$ and $S_{0,m,n-m}$.
So as explained in Subsection \ref{ss:compatibility} in general, a choice of crepant resolutions in neighborhoods of the points $(x_1, \dots, x_1)$ and $(x_2, \dots, x_2)$ determine crepant resolutions in neighborhoods of the point in $S_{m, 0, n-m}$ for all $0 < m < n$. Note that we can check compatibility across $S_{n, 0, 0}$ and $S_{0, 0, n}$ in the diagonal stratum $\{ (x, x, \dots, x) : x \in S \}$ as it is the join $S_{n, 0, 0} \vee S_{0, 0, n}$. The diagonal stratum has a unique crepant resolution given by the Hilbert scheme, so all choices of crepant resolutions glue compatibility in the diagonal. 

We have seen that an arbitrary choice of local crepant resolutions of the minimal strata $\Sym^n \{z_i\}$  determine at most one local resolution at all minimal strata. They therefore determine at most one local crepant resolution  at all strata.  We also saw that these resolutions cannot have any monodromy. To conclude that there is a unique locally projective crepant resolution determined by our choice of local resolutions at the strata $\Sym^n \{z_i\}$, it remains to show that there exists a compatible choice of local crepant resolutions for all the strata.  Each stratum, as above, is of the form $S \times \prod_i \Sym^{m_i} \{z_i\}$.  The slice here is a product of slices to $S \subseteq \Sym^{n-\sum_i m_i} V$ and to $\Sym^{m_i} \{z_i\} \subseteq \Sym^{m_i} U_i$. The former has a unique crepant resolution and the latter resolutions are uniquely determined by our resolutions at the strata $\Sym^n \{z_i\}$. So again as explained in Subsection \ref{ss:compatibility}, we obtain a resolution for this stratum, and by construction these are all compatible. This determines a unique locally projective crepant resolution of $\Sym^n Y$.

Finally, we show that
the locally projective resolutions are globally projective, assuming $Y$ has finitely many singularities.
 Let $Z := \bC^2/G$ be a du Val singularity, and let $\widetilde Z$ be its minimal resolution (which is crepant). Recall that the Hilbert-Chow morphism $\Hilb^{n}(\widetilde Z) \rightarrow \Sym^n(\widetilde Z)$ is a symplectic resolution of singularities. 
There is a map $\text{exc}: \Pic(\Hilb^{n}(\widetilde Z)) \rightarrow \bZ$ projecting to the exceptional divisor. For different choices of movable
bundle at the most singular points to glue to a section of $\cP^\sharp$ (and to fit into an M-configuration) we need their images under $\text{exc}$ to agree. This can be arranged e.g., by rescaling an arbitrary collection, so arbitrary choices of movable bundles at the most singular points glue. 

By the same argument as for $\ccR$, there is no monodromy of $\cP^\sharp$, unless there is on strata of $\Sym^m V$. Using $M$-configurations we can reduce to the case of the diagonal stratum $V \subseteq \Sym^m V$. But there the local relative Picard group is $\bZ$, spanned by the globally defined exceptional divisor of the Hilbert--Chow resolution. So the local sections of $\cP$ on the diagonal stratum all extend, and there is no monodromy. 

We thus conclude that $\Gamma(X,\cP^\sharp)$  is isomorphic to the direct sum of $\bZ$ (for the diagonal stratum) with the sum over all singularities $z \in Z$ of the kernel of $\text{exc}$ in the local sections of $\cP$ at $n\cdot z$. This is spanned by global divisors: the exceptional one together with the strict transform of the exceptional
divisors of the minimal
resolution of $Z$. The exceptional divisor is supported on all of $X$, whereas the strict transforms of exceptional divisors of $\tilde Z$ at a du Val singularity $z\in Z$ are supported on the closed codimension-two stratum of schemes which intersect $z$. Linear combinations are supported on the corresponding union of these closed strata. Thus, Lemma \ref{l:glob-crit} and Proposition \ref{p:r-surjective} apply, and all global sections of $\cP^\sharp$ are classes of global line bundles. Hence all locally projective partial crepant resolutions are globally projective. \qedhere



\end{proof}

\begin{rem}
    Since symmetric powers of du Val singularities are quotient singularities, in the terminal case, we could also have applied Theorem \ref{t: quotient-sing-globally-projective} to replace the argument in the last paragraph.  
\end{rem}

\subsection{Hilbert schemes of a surface with du Val singularities}
Let $\Gamma \subset \text{SL}_2(\bC)$ be a non-trivial, finite subgroup and let $n >1$. The $n$th symmetric power $\Sym^n(\bC^2/\Gamma)$ has multiple non-isomorphic symplectic resolutions of singularities. One such resolution is given by 
\[
n\Gamma\text{-}\Hilb(\bC^2) := \{ \Gamma\text{-invariant ideals } I \subset \bC[x, y] \mid  \bC[x,y]/I \cong \bC[\Gamma]^{\oplus n} \text{ as } \Gamma\text{-representations} \}.
\] 
where $\bC[\Gamma]$ denotes the regular representation. Additionally, $\Sym^n(\bC^2/\Gamma)$ has a \emph{partial} resolution of singularities given by 
\[
\Hilb^{n}(\bC^2/\Gamma) := \{ \text{ideals } I \subset \bC[x, y]^{\Gamma} \mid \dim_{\bC}(\bC[x,y]^{\Gamma}/I) = n \}. 
\]
Craw--Gammelgaard--Gyenge--Szendr\H{o}i \cite[Theorem 1.1]{CGGS19} and Craw--Yamagishi \cite[Theorem 1.2]{CY-Hilbert} prove that the invariants map
\[
n\Gamma\text{-}\Hilb(\bC^2) \rightarrow \Hilb^{n}(\bC^2/\Gamma) \hspace{1cm} I \mapsto I \cap \bC[x, y]^{\Gamma}
\]
is the unique projective symplectic resolution of singularities for $\Hilb^{n}(\bC^2/\Gamma)$.

Now let $Y$ be a surface with finitely many du Val singularities and let $X := \Hilb^{n}(Y)$ be its Hilbert scheme. By assumption, for each $z_i \in Y^{\sing}$ there is an analytic neighborhood $U_{i}$ which is either smooth or du Val, locally isomorphic to $\bC^2/\Gamma_i$. In each case, for every $m \geq 1$, $\Hilb^{m}(U_i)$, which is smooth if $U_i$ is smooth by \cite[Theorem 2.4]{Fogarty68}, has a unique symplectic resolution of singularities given by $m\Gamma_i\text{-}\Hilb(\tilde U_i) \to \Hilb^m(U_i)$ for some open $\tilde U_i \subseteq \bC^2$ with $U_i \cong \tilde U_i/\Gamma_i$. By the same argument as in the previous subsection, these glue uniquely to a locally projective symplectic resolution of $\Hilb^n(Y)$:  any locally projective crepant resolution is uniquely determined by its restriction to $\Hilb^n (V \sqcup (\sqcup_i U_i))$, and the preceding defines compatible locally projective crepant resolutions for the connected components,  $\Hilb^{m-\sum_i m_i} (V) \times \prod_i \Hilb^{m_i} U_i$.  Call this resolution $\tilde X \to X := \Hilb^n Y$.

Composing these with the Hilbert--Chow maps $\Hilb^{m_i} U_i \to \Sym^{m_i} U_i$ produces one of the locally projective crepant resolutions of $\Sym^{m_i} U_i$ constructed in the preceding subsection, which we showed to actually be globally projective. A relatively ample line bundle for the resulting resolution $\tilde X \to \Sym^n Y$ is automatically relatively ample over $\Hilb^n Y$. So we see that $\tilde X \to \Hilb^n Y$ is also globally projective. By construction it is the unique globally projective crepant resolution, since it restricts to the unique one for neighborhoods $\Hilb^n U_i$.

We conclude:
\begin{prop}
There is a unique globally projective crepant resolution of $\Hilb^n Y$, given by gluing the resolutions $m\Gamma_i\text{-}\Hilb \tilde U_i \to \Hilb^m U_i$.  In the case that $Y$ is a global quotient $Y \cong \tilde Y / \Gamma$, this resolution is isomorphic to $n\Gamma\text{-}\Hilb(\tilde Y)$.
\end{prop}

\begin{rem}
It follows from the above that, if $X = \Hilb^{n}(Z)$ for some surface $Z$, then $X$ admits a crepant resolution precisely when $Z$ has at worst du Val singularities. The preceding gives the existence of a unique projective crepant resolution when $X$ has at worst du Val singularities.  Conversely, note that all of the singularities of $Z$ appear in $X$, and already in the set of distinct $n$-tuples of points of $Z$. But it is well known that the canonical surface singularities are precisely the du Val singularities. In order to have a crepant resolution, the singularities must be canonical.
\end{rem}
\begin{rem}
    We did not check whether there exists non-projective, locally projective crepant resolutions of $\Hilb^n Y$. This seems like an interesting problem. We claim any such resolution $\pi$ of $\Hilb^n Y$, gives a proper crepant resolution $\pi'$ of $\Sym^n Y$ that is \emph{not} locally projective. Explicitly $\pi'$ is the composition of $\pi$ with the Hilbert--Chow morphism. If $\pi'$ is locally projective, then the final sentence of Proposition \ref{p:sym-duval} implies that $\pi'$ is globally projective. It follows that $\pi$ is globally projective and hence isomorphic to the unique globally projective crepant resolution constructed above. 
\end{rem}

\subsection{Multiplicative quiver varieties} \label{ss:MQV}

Let $Q$ be a quiver, $k$ an algebraically closed field of characteristic zero, and $A$ a $k$-algebra with a $kQ_0$-bimodule structure. Notice that an $A$-module $M$ has a dimension vector $d := (d_i) := ( \dim(e_i \cdot M)) \in \bN^{Q_0}$. The group $G_d := \prod_i \GL_{d_i}(k)$ acts on the space of $A$-modules by conjugation. Following King \cite{King94}, define $\theta \in \Hom_{\text{grp}}( G_d, k^*) = \bZ^{Q_0}$ to be a character of $G_d$. The quiver variety for $A$ is the moduli space $\cM_{d, \theta}(A)$ of  $d$-dimensional, $\theta$-semistable representations of $A$.

If $A = \Pi^{\lambda}(Q)$ is the deformed preprojective algebra of the quiver $Q$, then $\cM_{d, \theta}(A)$ is the Nakajima quiver variety. In this case we write $\cM_{\lambda, \theta}(Q, d) := \cM_{d, \theta}( \Pi^{\lambda}(Q))$ in order to conform to the more standard notation. 
If $A = \Lambda^q(Q)$ is the \emph{multiplicative preprojective algebra} of Crawley-Boevey and Shaw \cite{CBS06}, then $\cM_{d, \theta}(A)$ is the \emph{multiplicative quiver variety}. 

Multiplicative quiver varieties includes character varieties of Riemann surfaces with monodromy conditions (as open subsets, which are the entire variety in genus zero), and modifications of du Val singularities. We showed in \cite[Theorem 5.4]{KS20} that the formal local structure of multiplicative quiver varieties agrees with that of Nakajima quiver varieties.   Consequently, for each multiplicative quiver variety $\cM_{d, \theta}(\Lambda^q(Q))$ and each module $M$ one can classify symplectic resolutions for a sufficiently small open neighborhood $U_M$ of $M$. Moreover, the slices to the strata are also Nakajima quiver varieties, and whenever variation of stability produces a symplectic resolution or relative minimal model, then the slice is also resolved, verfying (*). By \cite{BS16}, this is always true except in the ``(2,2)'' case, where a further blow-up of the reduced singular locus produces a crepant resolution of singularities. It follows that all multiplicative quiver varieties have stratifications satisfying (*) (which is also a consequence of \cite{BCHM10}, see Remark \ref{r:bchm}).

The main results of this paper give a theoretical technique to classify global symplectic resolutions, provided one can compute the symplectic leaves and calculate the monodromy and compatibility constraints. Without monodromy computations, we can still prove partial results.

For dimension vectors $d, d' \in \bN^{Q_0}$ we write $d' < d$ if $d - d' \in \bN^{Q_0}$. The decomposition type of a semisimple representation is the unordered collection of dimension vectors of simple summands with multiplicities. Note that every point of a moduli space $\cM_{d,0}(A)$ of a quiver algebra $A$ is represented by a unique semisimple representation up to isomorphism. Taking connected components of the loci with fixed decomposition type, we obtain a finite stratification of $\cM_{d,0}(A)$ by connected strata (although these strata need not be smooth).  
\begin{prop} \label{prop:mqv}
Let $A$ be an augmented $kQ_0$-algebra, i.e., $A$ is equipped with a $kQ_0$-algebra homomorphism $A \twoheadrightarrow kQ_0$.
Equip $\cM_{d,0}(A)$ with the stratification above. Let the zero representation of a given dimension vector denote the one factoring through the augmentation. 
The following conditions are equivalent:
\begin{itemize}
\item[(1)] The zero representation lies in the closure of every stratum (i.e., symplectic leaf) of $\cM_{d,0}(A)$.
\item[(2)]  The variety $\cM_{d',0}(A)$ is connected for all $d' < d$.
\end{itemize} 
\end{prop}

\begin{proof}
(1) $\implies$ (2): By definition, every stratum is connected. Since the closure of a connected set is connected, and the union of intersecting connected sets is connected, we establish that $\cM_{d,0}(A)$ is connected. Since $\cM_{d',0}(A)$ appears as a union of strata in $\cM_{d, 0}(A)$ for all $d' < d$, we have that all such multiplicative quiver varieties are connected.  \\
(2) $\implies$ (1): By contrapositive, assume that the closure of some stratum of 
does not contain the zero representation. This stratum is a product of (Zariski) open strata in quiver varieties of dimension vectors adding to $d$, corresponding to the decomposition type of the representations in the stratum. Thus for some $d'<d$, there is an open stratum whose closure does not contain the zero representation. Assume $d'$ is minimal with this property. Take the union $X$ of all such closures, and let $Y$ be the union of closures of open strata containing the zero representation. We claim that $X$ and $Y$ are open, so the
moduli space is disconnected. Since these are closed sets whose union is the whole variety, it suffices to show that $X \cap Y = \emptyset$.  If not then the intersection contains a non-open stratum whose closure does not contain the zero representation. This contradicts minimality of $d'$.
\end{proof}
We apply the preceding result in the case of $A=\Lambda^1(Q)$. We need to take $q=1$ for the relation $\prod_{a \in Q_1} (1+aa^*)(1+a^*a)^{-1} - q$ to lie in the ideal $(k Q_1)$, which we take to be the augmentation ideal. Note that, in general, multiplicative quiver varieties may be reducible or even disconnected. There is no trouble applying our theory to this case (and indeed we did not assume irreducibility), as birational morphisms make sense in the reducible setting.
\begin{cor} \label{cor:mqv_res}
Suppose that $d$ is a dimension vector such that $\cM_{d', 0}(\Lambda^1(Q))$ is connected for all $d' <d$. Then $\ccR(\cM_{d, 0}(\Lambda^1(Q))) \hookrightarrow \ccR(\cM_{d, 0}(\Pi(Q)))$.
\end{cor}

More explicitly, the map is given by restricting to a neighborhood of the zero representation, identifying with a neighborhood of the zero representation in an additive quiver variety and extending to the entire variety. 

\begin{proof}
    It suffices to establish the hypotheses of Corollary \ref{cor:unique_minimal} in the case $X =\cM_{d, 0}(\Lambda^1(Q))$, $s$ is the zero representation, and $Y = \cM_{0, 0}(Q, d)$. By Proposition \ref{prop:mqv} the zero representation $s \in \cM_{d, 0}(\Lambda^1(Q))$ is contained in the unique minimal stratum. Additionally, there exists a neighborhood of zero in $\cM_{d, 0}(\Lambda^1(Q))$ that is isomorphic to a neighborhood of zero in $\cM_{0, 0}(Q, d)$ by \cite[Corollary 5.22]{KS20} (noting that $Q' = Q$ and $d' =d$ in this case since the Ext-quiver of the zero representation is the original quiver). So Corollary \ref{cor:unique_minimal} gives the inclusion $\ccR(\cM_{d, 0}(\Lambda^1(Q)))\hookrightarrow \ccR(\cM_{0, 0}(Q, d))$, completing the proof. 
\end{proof}

The original motivation for this paper was to extend the classification of symplectic resolutions of quiver varieties \cite{BS16, BCS22} to multiplicative quiver varieties. For quiver varieties the symplectic leaves are given by representation type. Multiplicative quiver varieties have a stratification by representation type but these strata need not be connected and the symplectic leaves may be a finer stratification. In examples below, we establish that the stratifications into symplectic leaves and representation types agree. Hence the classification of symplectic resolutions of the multiplicative quiver variety $\cM_{d, 0}(\Lambda^1(Q))$ agree with the known classification of symplectic resolutions of the quiver variety $\cM_{0, 0}(Q, d)$.\\


\underline{Type $\tilde{A}$}
For $Q = \widetilde{A}_{n-1}$, the cycle with $n$-vertices, $q=1$, $d=1$, and $\theta = 0$, our previous work \cite[Corollary 6.14, Proposition 6.19]{KS20} building on Shaw \cite[Theorem 4.1.1]{Shaw05} describes $\cM_{1, 0}(\Lambda^1(Q))$ as the spectrum of the commutative ring $R_n := k[X, Y, Z]/(Z^{n} + XY + XYZ)$. Note that $(1+Z)$ is invertible in $R_n$ as the relation can be rewritten:
\begin{align*}
Z^n + XY + XYZ &= Z^n + XY(1+Z) = Z^n + 1 -1 + XY(1+Z) = (Z^n+1) + XY(1+Z) - 1 \\
&= (1+Z)( Z^{n-1}- Z^{n-2} + \cdots + (-1)^{n-1} + XY) - 1.
\end{align*}
Hence there is an isomorphism of rings
\[
k[X, Y, Z][(1+Z)^{-1}]/(Z^n + XY)  \rightarrow k[X, Y, Z]/(Z^{n} + XY + XYZ)\  \ X \mapsto X(1+Z), \ Y \mapsto Y, \ Z \mapsto Z.
\]
Since the vanishing of $Z^n + XY$ is the du Val singularity for $\widetilde{A}_{n-1}$, we conclude that this multiplicative quiver variety is a Zariski-open subset of the (additive) quiver variety. In particular, $\cM_{1, 0}(\Lambda^1(Q))$ has a unique symplectic resolution of singularities. 


\begin{exam}
    Let $Q = \tilde{A_2}$, $d = (2, 3, 5)$, and $\lambda = 0 = \theta$. Then the canonical decomposition of $d$ is 
    \[
    (2, 3, 5) = 2(1, 1, 1) + (0, 1, 0) + 3(0, 0, 1).
    \]
    The decomposition on the level of quiver varieties, is
    \begin{align*}
    \cM_{0, 0}( Q, d) & \cong \Sym^{2}( \cM_{0, 0}(Q, (1, 1, 1))) 
    \times  \cM_{0, 0}(Q, (0, 1, 0))
    \times \Sym^{3}( \cM_{0, 0}( Q, (0, 0, 1))) \\
    &\cong \Sym^{2}( \cM_{0, 0}( Q, (1, 1, 1))) 
    \end{align*}
    since $\cM_{0, 0}(Q, e_i)$ is a single point, for any elementary vector $e_i = (0 ,\dots, 0, 1, 0, \dots, 0)$. This realizes 
    \[
    \cM_{0, 0}(Q, d) = \Sym^2( \cM_{0, 0}(Q, (1, 1, 1))) = \Sym^2( \bC^2/C_2 )
    \]
    as a symmetric product of a surface with a du Val ($A_1$) singularity. Hence, $\cM_{d, 0}(\Pi(Q))$ has two symplectic resolutions of singularities given by $\Hilb^{2}( \widetilde{\bC^2/C_2})$ and $2 C_2$-$\Hilb(\bC^2)$.

    We claim that the multiplicative quiver variety $\cM_{d, 0}(\Lambda^1(Q))$ has the same classification of symplectic resolutions. In fact, it has the same stratification into symplectic leaves by representation type and hence is a Zariski-open subset 
    \[
    \cM_{d, 0}(\Lambda^1(Q)) = \Sym^2( \cM_{1, 0}(\Lambda^1(Q)) \subset \Sym^2( \bC^2/C_2 )
    \cong  \cM_{d, 0}( \Pi(Q)).
    \]
\end{exam}

We conclude that all symplectic resolutions of this multiplicative quiver variety are given by variation of geometric invariant theory quotient (VGIT), since the same holds in the additive case by Bellamy--Craw \cite[Corollary 1.3]{BC20}. \\

\underline{Hyperpolygon spaces}
Let $Q_n$ be the star-shaped quiver with a central vertex $v$, $n$ external vertices $w_1, \dots, w_n$, and an arrow from each external vertex to the central vertex. Let $d$ be the dimension vector $(2, 1, 1, \dots, 1)$ where the vertices are ordered $(v, w_1, \dots, w_n)$. The Nakajima quiver varieties $\cM_{0, 0}(Q_n, d)$ for these data are called \emph{hyperpolygon spaces} due to being a hyperk\"ahler analogue of moduli spaces of polygons in $\bR^3$, see \cite{HausmannKnutson}. In \cite[Corollary 1.4]{BCRSW21}, the authors show that all crepant resolutions of  the hyperpolygon spaces are given by variation of GIT and consequently can be counted explicitly in terms of GIT chambers. Further \cite[Proposition 3.8, Remark 3.9]{BCRSW21} counts the crepant resolutions for $n \leq 9$.
 We claim that these results are valid in the multiplicative setting.

\begin{prop} \label{prop:hyperpolygon}
There is a bijection $\ccR(\cM_{d, 0}(\Lambda^1(Q_n))) \cong \ccR(\cM_{0, 0}(Q_n, d))$ given by the inclusion map from Corollary \ref{cor:mqv_res}. In particular, \cite[Remark 3.9]{BCRSW21} applies to give the number of locally projective crepant resolutions, which are all globally projective. 
\end{prop}
As stated above, $n=4, 5, 6$, $\cM_{d, 0}(\Lambda^1(Q_n))$ has 1, 81, and 1684 locally  projective crepant resolutions respectively, which are all globally projective.

\begin{lem} \label{lem:irred}
    Fix $n, m \in \bN$ with $m,n>1$ and  monic polynomials $\chi_i \in \bC[x]$ of degree $m$. Let $X  \subseteq \GL_m(\bC)^n$ the subset of matrices such that the $i$-th matrix has characteristic polynomial $\chi_i$ for all $i$.  Let $Y_c \subseteq \GL_m(\bC)$ denote the subset of matrices of trace $c$. Define the locus of $n$-tuples of $m \times m$ matrices
    \[ 
    Z := \{  (A_1, A_2, \dots, A_n,B) \in X \times Y_c \ | \  
    A_1 A_2 \cdots A_n B = I \}.
    \]
    Then, $Z$ is irreducible.
\end{lem}

\begin{proof}
First, note that the variety embeds into the subset of $X$ of $n$-tuples such that the inverse of the product has a fixed value $c$ of the trace. 

Consider matrices $A_1, A_2, \ldots, A_n$ with Jordan decompositions of the inverses $A_i^{-1} := S_i+N_i$, where $S_i$ is semisimple and $N_i$ is nilpotent. For each $i$ and $\lambda \in \bC$ consider 
\[
X_{i,\lambda}:= (S_n+N_n) \cdots (S_{i-1}+N_{i-1}) (S_i+\lambda N_i) (S_{i+1}+N_{i-1}) \cdots (S_1+N_1)
\]
rescaling the nilpotent part of the $A_i^{-1}$. Now $\tr(X_{i,\lambda})$ is a linear function of $\lambda$, so either it obtains all values once, or it is constant. The condition to be constant, $\tr(X_{i,\lambda})= \alpha$ a fixed value, is closed on the $n$-tuples of matrices, hence also on $X$. 
It is not the entire space since we can take all $S_1,\ldots, S_n$ to be diagonal, take two matrices $N_i, N_j$ to have $\tr(N_i N_j) \neq 0$, and all other $N_k$ with $k \notin \{i,j\}$ to be zero. 

So, we get that on a dense open subset $U \subset X$, 
the locus with $\tr(X_{i,\lambda})= \alpha$ is given by a unique choice of $\lambda$. So the map
\[
 U \times \bC^\times \rightarrow \bC, \quad (X,\lambda) \mapsto \tr(X_{i,\lambda})
\]
has fibers which project 
isomorphically to $U$. In particular, the fibers are irreducible. 

Applying the multiplication map $U \times \bC^\times \rightarrow U$, the image of this fiber gives an open dense irreducible set of $n$-tuples which have the given value of trace. The whole locus of desired tuples is thus irreducible. 
\end{proof}

\begin{proof}[Proof of Proposition \ref{prop:hyperpolygon}]
    To obtain the inclusion map of Corollary \ref{cor:mqv_res}, we need to establish that $\cM_{d', 0}(\Lambda^1(Q_n))$ is connected 
    for each $d' < d$. This is a consequence of Lemma \ref{lem:irred} for $m=2$ and characteristic polynomials equal to $(x-1)^2$, since  a product of unipotent $2 \times 2$ matrices is unipotent if and only if it has trace two.

    For the injection to be surjective, we need to show that every local crepant resolution near zero extends to a global one. We will show these are in fact given by VGIT and hence globally projective, proving simultaneously the last statement.
     

    In \cite[Theorem 5.4]{KS20}, we establish that $\cM_{d, 0}(\Lambda^1(Q_n))$ is locally a quiver variety. Note that the hypotheses of \cite[Theorem 1.5]{ST19} are satisfied,
    so $\cM_{d, 0}(\Lambda^1(Q_n))$ has a crepant resolution given by variation of stability parameter (for a generic choice of parameter). 
    The varieties $\cM_{d, 0}(\Lambda^1(Q_n))$ and $\cM_{0, 0}(Q, d)$ have the same choices of stability parameter and the same symplectic leaves since the stratification by representation type is connected. It follows that the injection of resolutions $\ccR(\cM_{0, 0}(Q, d)) \hookrightarrow \ccR(\cM_{d, 0}(\Lambda^1(Q_n)))$ is surjective and we have the same classification of symplectic resolutions.  \qedhere

\end{proof}
\begin{rem}
    According to the interpretation of multiplicative quiver varieties as character varieties \cite{CBS06}, \cite[Theorem 3.6]{ST19}, the multiplicative quiver varieties above can also be interpreted as the moduli spaces of rank two local systems on Riemann spheres punctured at $n$ points with unipotent monodromies about the punctures. Thus Proposition \ref{prop:hyperpolygon} gives a complete classification of the (locally) projective crepant resolutions of these character varieties. One can deduce from this that distinct crepant resolutions are connected by sequences of local Mukai flops, precisely as in \cite{BCRSW21}. 
\end{rem}
The proof of  Proposition \ref{prop:hyperpolygon} holds in the following more general case:
\begin{cor}
    Suppose $Q$ is a quiver and $d$ is a dimension vector satisfying:
    \begin{itemize}
        \item $\cM_{d', 0}(\Lambda^1(Q))$ is connected for all $d' < d$,
        \item The hypothesis of \cite[Theorem 1.2]{BCS22} holds for 
        $\cM_{0, 0}(Q, d)$ showing it is a Mori dream space with all 
        projective crepant  resolutions given by VGIT.
    \end{itemize}
    Then $\ccR(\cM_{d, 0}(\Lambda^1(Q))) \cong \ccR(\cM_{0, 0}(Q, d))$ and all resolutions are given by VGIT (hence globally projective). 
\end{cor}
Note here that the hypotheses of \cite[Theorem 1.5]{ST19} for the multiplicative variety $\cM_{d, 0}(\Lambda^1(Q))$ to admit a symplectic resolution by VGIT are automatically satisfied when the conditions of \cite[Theorem 1.2]{BCS22} hold for $\cM_{0,0}(Q,d)$, since the latter include the facts that the dimension vector is indivisible and contained in $\Sigma_{0,0}$.

\subsection{Moduli spaces of objects in 2-Calabi--Yau categories}

A $k$-linear triangulated category with finite-dimensional hom spaces, $\bf{C}$, is 2-Calabi--Yau if it has a Serre functor given by shift by 2. Observe that 2-Calabi--Yau categories arise naturally in geometry including any full subcategory of:
\begin{itemize}
\item[(1)] the derived category of finite-dimensional modules for a 2-Calabi--Yau dg-algebra, which includes dg-preprojective algebras and multiplicative dg-preprojective algebras (see e.g., \cite{BCS21}),
\item[(2)] the derived category of coherent sheaves 
on a compact Calabi--Yau surface (K3 or abelian variety);
\item[(3)] the wrapped Fukaya category of a real 4-dimensional Weinstein manifold \cite{Ganatra12},
\item[(4)] the cluster category of a finite quiver (see e.g., \cite{Keller08}), and 
\item[(5)] the category of semistable Higgs bundles of fixed slope on a closed Riemann surface $M_{g}$.
\end{itemize}
The Higgs bundle example follows from realizing this category is a subcategory of coherent sheaves on the symplectic manifold $T^*(M_g)$. Note that, analogously, a full subcategory of (1) for the multiplicative preprojective algebra includes categories of local systems on Riemann surfaces with punctures, fixing monodromy conjugacy classes at the punctures (see \cite[Theorem 3.6]{ST19} and the preceding discussion).

The key idea in this section is to build crepant resolutions for a variety $X$ by realizing $X$ as the moduli space of objects for $\bf{C}$, a 2-Calabi--Yau category, $X \cong \cM(\text{ob}(\bf{C}))$. Davison proved that such moduli spaces are \'etale-locally quiver varieties \cite[Theorem 5.11, Theorem 1.2]{Davison21}, at least over an open subset: Given a closed point $x \in X$ corresponding to a ``simple-minded collection'' of objects in $\mathbf{C}$, i.e., a direct sum of objects $\mathcal{F}_i^{\oplus r_i}$ such that $\dim \Hom(\mathcal{F}_i, \mathcal{F}_j) = \delta_{ij}$ and $\Ext^{<0}(\mathcal{F}_i, \mathcal{F}_j) = 0$ for all $i,j$, there is an \'etale neighborhood of 
$x \in X$ which is isomorphic to an \'etale neighborhood of the zero representation in a conical Nakajima quiver variety.  These also have a symplectic form on the smooth locus 
\cite[Theorem 5.5]{BD-RelCY2}, so that the (open locus of simple-minded objects in) the coarse moduli space forms a symplectic singularity. The argument in the case of multiplicative quiver varieties then shows that the stratification by symplectic leaves satisfies (*).

Bellamy and the second-named author determine which quiver varieties admit projective, crepant (or equivalently symplectic) resolutions in \cite{BS21}. Recently, Bellamy, Craw, and the second-named author give a criterion for a variety $X$ presented as a GIT quotient to have all projective crepant resolutions given by variation of GIT \cite[Theorem 1.1, Condition 3.4]{BCS22}. Moreover, they prove that, under mild hypotheses, Nakajima quiver varieties satisfy this criterion, thus giving a classification of projective crepant resolutions by VGIT \cite[Theorem 1.2]{BCS22}.

Consequently, for $\cM(\text{ob}(\bf{C}))$ the moduli of objects in a 2-Calabi--Yau category, one can compute the local classification of projective, crepant resolutions in the \'etale neighborhood of any point, by first identifying the neighborhood with that of a quiver variety, where every such resolution can be described by VGIT. 

In more detail, consider the open subset of objects $M$ with $\End(M)$ semisimple. Then there exists a quiver $Q = (Q_0, Q_1)$ with (1) $\End(M)$ Morita equivalent to $kQ_0$ and (2) the Morita equivalence taking $\Ext^1(M, M)$ to $k Q_1$. The quiver $Q$ is called the Ext-quiver for $M$. And in the 2-Calabi--Yau case the entire Ext-algebra $\Ext^*(M, M)$ can be recovered (as a graded vector space) from $Q$ since $\Ext^2(M, M) \cong \Ext^0(M, M)^*$ and $\Ext^n(M, M) = 0$ for $n>2$. Note further that $\Ext^1(E_i, E_j) \cong \Ext^1(E_j, E_1)^*$ so $\dim(\Ext^1(E_i, E_j)) = \dim(\Ext^1(E_i, E_j))$ and hence $Q$ is a doubled quiver. The multiplication on $\Ext^*(M, M)$ is determined by the pairing on $\Ext^1(M, M)$ (together with vanishing in degree $\geq 3$). While, in general one would need to describe the entire $A_{\infty}$-structure on $\Ext^*(M, M)$, Davison \cite[Theorem 1.2]{Davison21} proved that this algebra is formal. 

By computing monodromy and compatibility, one could apply our local-to-global analysis, where the stratification $\cS$ is determined by Ext-quiver type (passing to connected components), to classify global projective symplectic resolutions of $X$.

Even without computing monodromy and compatibility we can still apply Corollary \ref{cor:unique_minimal} 
to obtain:

\begin{cor}
    Let $\bf{C}$ be a 2-Calabi--Yau category and let $\cM(\text{ob}(\bf{C}))$ denote its moduli space of objects. Assume $\cM(\text{ob}(\bf{C}))$ has a unique minimal stratum with basepoint $M$. Let $Q$ be the Ext-quiver for $M$ and $d = (d_i := \dim(\Ext^1( e_i M, e_i M)))$. Then
    \[
        \ccR(\cM(\text{ob}(\bf{C}))) \hookrightarrow \ccR(\cM_{0, 0}(Q, d)).
    \]
    More generally, we always have an embedding $\ccR(\cM(\text{ob}(\bf{C}))) \to \prod_i \ccR(\cM_{0, 0}(Q_i, d_i))$  into the product over all minimal strata of the set of local resolutions near a fixed basepoint of the stratum.
\end{cor}

\subsection{Finite symplectic quotients of symplectic tori}
Let $\bT = (\bC^\times)^n$ be the $n$-dimensional complex torus, and let $\mathfrak{t} \cong \bC^n$ be its Lie algebra.  To this we can associate two natural symplectic varieties: $T^*( \bT) \cong \bT \times \mathfrak{t}^*$ and $\bT \times \bT$.  

As explained in Section \ref{s: general results}, any finite symplectic quotient has a finite stratification by symplectic leaves, with a local product decomposition.   In this section we will look at the case where the actions are moreover group automorphisms of the torus. 
Recall that the group of multiplicative automorphisms of the torus $\bT$ is the group of automorphisms of the lattice $\Hom(\bT, \bC^\times) \cong \bZ^n$ of characters, i.e., $\GL_n(\bZ)$. These induce symplectic automorphisms of $T^*(\bT)$ and $\bT \times \bT$. We will also consider the larger group $\Sp_{2n}(\bZ)$ of symplectic group automorphisms of $\bT \times \bT$.

\subsubsection{Weyl group quotients} \label{ss:Weylgroup}
Let $G$ be a reductive group with maximal torus $\bT$ and Weyl group $W := N_G(\bT) / \bT$.  Here we consider the quotients $\bT \times \bT / W$ and $T^*(\bT) / W$. 

\begin{rem} 
  The three classes of quotients (1) $\bT \times \bT / W$,  (2) $T^*(\bT) / W$, and (3) $T^*(\mathfrak{t}) / W$ are of interest in part because of their descriptions as the connected component of the identity of the quasi-Hamiltonian reductions: (1)  $G \times G /\!/\!/ G$, (2) $T^*(G) /\!/\!/ G$, and (3) $T^*(\mathfrak{g}) /\!/\!/ G$. 
  These statements can be found in \cite[Theorem 1.1.2, Theorem 1.1.4]{LNY23} but note Ansatz 1.2.5 and see Remark 1.1.5 for previous partial results including \cite{Joseph} and \cite{EG02} who obtain the result only after quotienting out the nilradical, i.e., for reduced rings. 
  One could use the formula for $G \times G /\!/\!/ G$ in \cite[Theorem 1.1.4]{LNY23} to analyze the other connected components, which are quotients of smaller tori.
\end{rem}

Let us first  consider the case where $G$ is simply-connected and semisimple. Let  $\fS_n$ denote the symmetric group on $n$ letters and $C_n$ denote the cyclic group of order $n$. 

\underline{Type A ($\SL_{n+1}$)}: The quotient $\bT \times \bT/ W := (\bC^\times)^n \times (\bC^\times)^n / \fS_{n+1}$ can be viewed as the subset of $\Sym^{n+1}(\bC^2)$ given by 
\[
\left \{ (\lambda_1, \mu_1), (\lambda_2, \mu_2), \dots, (\lambda_{n+1}, \mu_{n+1}) \in \Sym^{n+1}(\bC^2) \ \middle | \ \prod_{i=1}^{n+1} \lambda_i = 1 = \prod_{i=1}^{n+1} \mu_i  \right \}.
\]
Explicitly, $\lambda_{n+1} = 1/ \prod_{i=1}^n \lambda_i$ and $\mu_{n+1} = 1/ \prod_{i=1}^n \mu_i$ and in particular each $\lambda_i, \mu_i \neq 0$. Hence each singularity in $\bT \times \bT/W$ can be identified with a singularity in $\Sym^{n+1}(\bC^2)$. Consequently, each singular point $x$ has a neighborhood $U_x$ with a unique crepant resolution given by the Hilbert--Chow map $\Hilb^{n+1}(U_x) \rightarrow \text{Sym}^{n+1}(U_x)$. By Corollary \ref{c: unique locally-proj resolution}, these unique local resolutions glue to a unique global, crepant resolution. This resolution is a priori only locally projective but in this case is given by the global Hilbert scheme 
\[
\Hilb_1^{n+1}(\bC^\times \times \bC^\times) :=  \left \{ I \in \Hilb^{n+1}(\bC^\times \times \bC^\times) \ \middle | \ \prod_{i=1}^{n+1} \lambda_i = 1 = \prod_{i=1}^{n+1} \mu_i \right \}. 
\]

\underline{Type B ($\SO_{2n+1}$) and type C ($\Sp_{2n}$):}  Note that the reduced quotients in these cases are the same torus quotients, given as follows:
\[
\bT \times \bT / W = (\bC^\times)^n \times (\bC^\times)^n / (C_2^n \rtimes \fS_n) = \Sym^n(\bC^\times \times \bC^\times / C_2 )
\]
where $C_2$ acts on $\bC^\times \times \bC^\times$ by $(z, w) \mapsto (z^{-1}, w^{-1})$ and hence fixes the four points $(\pm 1, \pm 1)$. These are the four singularities of $\bC^\times \times \bC^\times / C_2$ and each singularity is an $A_1$ singularity $\bC^2 / C_2$, where $C_2$ acts on $\bC^2$ by $(z, w) \mapsto (-z, -w)$. 

Denote by $\ccR(X)$ the set of isomorphism classes of symplectic resolutions of the variety $X$. It follows that we obtain the following formula:
\[
\# \ccR(\bT \times \bT/ W) = ( \# \ccR(\Sym^n(\bC^2/ C_2)))^4 = n^4.
\]
Similarly, the number of symplectic resolutions for $T^*(\bC^\times)^n / (C_2^n \rtimes \fS_n)$ is $n^2$, the square of the number of symplectic resolutions of $\Sym^n(\bC^2/ C_2)$.\\ 
\\
In all other types, the formal completion of the quotient at the identity will recover the quotient $T^* \mathfrak{t} /\!/\!/ W$, which does not admit a symplectic resolution by \cite{Gordon-baby,Bellamy09}.  

Next consider the general type $A$ case.  Then the existence of a symplectic resolution was classified in \cite[Theorem 1.10]{BS19}.  In particular, for type $A_n, n \geq 2$, and $G$ (almost) simple, a symplectic resolution only exists if $G=\SL_{n+1}(\bC)$ as above.  
\begin{exam}\label{ex:pgl3}
For example, for the case $G=\PGL_3(\bC)$, we obtain a quotient $\bT \times \bT / \fS_3$, for $\bT$ a two-dimensional torus, which does not admit a symplectic resolution.  One explicit way to think about the action is to consider $\bT = (\bC^\times)^3 / \bC^\times \subset \PGL_3(\bC)$ with the quotient action being the diagonal action, then to act by usual permutations. Then the element $(123)$ fixes the nine points $Y \times Y$ for $Y = \{(1,\zeta,\zeta^2) \mid \zeta^3=1\}$. But the reflections $(12)$ and $(23)$ only fix the subset $Y \times \{1\}$ and $\{1\} \times Y$, respectively. This leaves $9-(3+3-1) = 4$ nonresolvable $C_3$-singularities.
\end{exam}

\subsubsection{$\bT \times \bT/ \Gamma$} \label{ss:symp_torus_quot}
In this section again $\bT \cong (\bC^\times)^{n}$ but now $W$ is replaced by $\Gamma \subset \Sp_{2n}(\bZ)$ a finite group acting on $\bT \times \bT$. We first restrict to the case $n=1$, and $\Gamma \cong C_i$ for $i= 2$, $3$, $4$ or $6$.

Define $\Gamma_i := \langle \gamma_i \rangle \cong C_i$ where $\gamma_i$ is given by
\[
\gamma_6 = \left ( \begin{array}{cc} 0 & 1 \\  -1 &  1 \end{array} \right ) \hspace{.8cm}
\gamma_4 = \left ( \begin{array}{cc} 0 & 1 \\  -1 &  0 \end{array} \right ) \hspace{.8cm}
\gamma_3 = \gamma_6^2 = \left ( \begin{array}{cc} -1 & 1 \\  -1 &  0 \end{array} \right ) \hspace{.8cm}
\gamma_2 = \gamma_4^2 = \left ( \begin{array}{cc} -1 & 0 \\  0 &  -1 \end{array} \right ). 
\]
$\Gamma_i$ acts on $\bC^\times \times \bC^\times$ via the weights matrix:
\[
\left ( \begin{array}{cc} a & b \\  c &  d \end{array} \right ) \cdot (x, y) 
= (x^a y^b, x^c y^d)
\]
so
\[
\gamma_6 \cdot (x,y)=(y, x^{-1} y)\hspace{1cm}\gamma_4 \cdot (x,y)=(y, x^{-1}) \hspace{1cm}
\gamma_3 \cdot (x,y)=(x^{-1}y, x^{-1})\hspace{1cm} \gamma_2 \cdot (x, y) = (x^{-1}, y^{-1}). 
\]
Each action preserves the symplectic form $\omega = dx/x \wedge dy/y$, as each weight matrix has determinant 1.
The fixed points of each action are
\begin{align*}
(x,y) = (y, x^{-1} y) &\implies x = y, \ y = x^{-1} y \implies x=y=1 \\
(x,y) = (y, x^{-1}) &\implies x = y, \ y =x^{-1} \implies x = y = \pm 1 \\
(x,y)=(x^{-1}y, x^{-1}) &\implies x^2 = y, \ y = x^{-1} \implies x = y^{-1} = \zeta_3 \\
(x, y) = (x^{-1}, y^{-1}) &\implies x = x^{-1}, \ y = y^{-1} \implies x = \pm 1, \  y = \pm 1.
\end{align*}
Each fixed point gives a singularity in the orbit space that is locally du Val and hence has a unique crepant resolution. These glue to a unique crepant resolution of the entire variety. Moreover, this resolution is projective and so the set of projective crepant resolutions of $(\bC^\times)^2/ \Gamma_i$ is a singleton. 

These are the only finite group actions on $(\bC^\times)^2$ preserving $dx/x \wedge dy/y$.  To see this, the following lemma is useful:
\begin{lem} \label{lem:non-integral_reps}
    Suppose that $G < \Sp_{2m}(\bR)$ with complexified symplectic representation $\bC^{2m}$ symplectically irreducible.  Then $\bC^{2m}$ is reducible as a complex representation. 
\end{lem}

\begin{proof}
     Observe that if a complex irreducible representation $V$ of a finite group is defined over $\bR$, it preserves a nondegenerate symmetric bilinear form (e.g., by averaging the standard inner product over the group), so it cannot preserve a symplectic form.
\end{proof}

\begin{prop} \label{prop:symp_tori_n=1}
     Let $\Gamma$ be a finite subgroup acting on $(\bC^\times)^{2}$ by group automorphisms preserving the invariant
    symplectic form $d x/x \wedge d y /y$. Then $\Gamma$ is cyclic of order $1,2,3,4$, or $6$ and the quotient  $(\bC^\times)^{2} / \Gamma$ admits a unique projective crepant resolution of singularities. 
\end{prop}

\begin{proof}
Such an action is given by a finite subgroup of $\SL_2(\bZ)$. These are cyclic of orders $1,2,3,4$, or $6$. An algebraic argument for this follows because $\SL_2(\bZ)$ is isomorphic to an amalgamated product $C_4 *_{C_2} C_6$. 

For a geometric argument, first by Lemma \ref{lem:non-integral_reps}, the group must act reducibly, hence it is conjugate to a subgroup of diagonal matrices therefore abelian. Next, the abelian finite subgroups of $\SL_2(\bC)$ are cyclic.  Finally, the cyclic subgroups with integral traces are precisely the ones of the given orders.  

For the final statement, all singularities of the quotient are du Val and hence admit unique local projective crepant resolutions of singularities.  So the quotient $(\bC^\times)^2/\Gamma$ admits a unique locally projective crepant resolution of singularities.  But also, this is globally projective, since it is well known that the resolutions can be obtained by iterated blowups of the (isolated) singularities.
\end{proof}

Using these examples, we can extend the type B/C Weyl group quotients as follows:
\begin{exam}
Let $W = \fS_n \ltimes \Gamma$ be a wreath product with $\Gamma < \Sp_2(\bZ)$ and $\fS_n$ permuting the pairs $(x_1, y_1), \dots, (x_n, y_n)$ and hence preserving the symplectic forms $\sum_{i=1}^n dx_i/x_i \wedge dy_i/y_i$ and $\sum_{i=1}^n dx_i \wedge dy_i$. So $W < \Sp_{2n}(\bZ)$.  In particular we can let $\Gamma = C_m$ for $m \in \{1,2,3,4,6\}$. Then the quotient $\bT \times \bT / W$ is identified with $\Sym^n(\bC^\times \times \bC^\times / \Gamma)$. In particular, there is a symplectic resolution given by $\Hilb^{n} Y$ for $Y$ the minimal resolution of $\bC^\times \times \bC^\times / \Gamma$. 
\end{exam}

The above example actually produces all of the groups whose quotient admits a projective  crepant resolution:

\begin{prop} \label{prop:symplectic_tori}
    Let $\Gamma \subset \text{Sp}_{2n}(\bZ)$ be a finite subgroup acting on $(\bC^\times)^{2n}$ by group automorphisms preserving the invariant
    symplectic form $\sum_{i} d x_i/x_i \wedge d y_i /y_i$. The quotient  $(\bC^\times)^{2n} / \Gamma$ admits a locally projective symplectic (equivalently crepant) resolution of singularities only if $\Gamma \cong \prod_i \Gamma_i$ with each of the $\Gamma_i$ of the form $\fS_n \ltimes H$ with $H \in \{C_1,C_2,C_3,C_4,C_6\}$, where the derivative of the action near $1 \in (\bC^\times)^{2n}$ is given by the product of the usual reflection representations of each $\Gamma_i$.
\end{prop}
\begin{proof}
Suppose that $(\bC^\times)^{2n}/\Gamma$ admits a crepant resolution for $\Gamma < \Sp_{2n}(\bZ)$. Then the local model at the identity is  $\bC^{2n}/\Gamma$ for the same group. Let us assume
     the representation $\bC^{2n}$ is symplectically irreducible, otherwise the pair $(\bC^{2n},\Gamma)$ decomposes as a product of pairs.
     Thus, by Lemma \ref{lem:non-integral_reps}, $\bC^{2n}$ is isomorphic to a sum of $G$-invariant Lagrangian subspaces $L \oplus L'$.  
If $\bC^{2n}/G$ admits a locally projective symplectic resolution, then $G$ must be generated by symplectic reflections \cite{Verbitsky00}.  In turn, if $G$ preserves a Lagrangian $L$, it must be a complex reflection group. Thanks to \cite{Bellamy09} building on  Ginzburg--Kaledin \cite[Corollary 1.2.1]{GK04} (and see also Etingof--Ginzburg \cite[Corollary 1.1.4 (i)]{EG02} and Namikawa \cite[Corollary 2.1]{Namikawa11}) we have a complete list of complex reflection groups $\Gamma$ such that the symplectic quotient singularity $\bC^{2n}/\Gamma$ admits a projective symplectic resolution. This includes the wreath products $\fS_n \ltimes \Gamma^n$ for $\Gamma < \SL_2(\bC)$ cyclic. In turn, as in 
 Proposition \ref{prop:symp_tori_n=1}, to have integral traces this means $\Gamma$ is cyclic of order $1,2,3,4$, or $6$, since for $\gamma \in \Gamma$, a corresponding reflection in the wreath product has trace $2n-2+\tr(\gamma)$. Other than the wreath products and $\fS_{n+1}$, the only other complex reflection group with symplectic quotient admitting a symplectic resolution is the binary tetrahedral group.  In Lemma \ref{l:bintet-nores} below, we show that the quotient of $(\bC^\times)^4$ by the binary tetrahedral group does not admit a symplectic resolution of singularities, completing the classification.
\end{proof}

For each group $\Gamma \subset \text{Sp}_{2n}(\bZ)$ appearing in Proposition \ref{prop:symplectic_tori}, there exists an action of $\Gamma$ on $(\bC^\times)^{2n}$ such that the quotient $(\bC^\times)^{2n}/\Gamma$ admits a symplectic resolution of singularities. Thus, Proposition \ref{prop:symplectic_tori} completes the classification of the rationalizations of integral representations such that the corresponding torus quotient admits a symplectic resolution of singularities. But a rationalization of an integral representation may admit multiple inequivalent integral forms. 
Moreover, for many fixed $\Gamma$, we can find an action on $(\bC^\times)^{2n}$ such that the resulting quotient does \emph{not} admit a symplectic resolution of singularities: see Example \ref{ex:pgl3}.
There can even be nontrivial ways to form an extension of two integral representations: 
\begin{exam} \label{exam:inequivalent_reps}
Let $C_4 = \langle i \rangle \subset \bC$ act on the lattice $\bZ^2 \cong \bZ_{1} \oplus \bZ_{i} \subset \bC$ by multiplication in $\bC$. The group $C_4 \times C_4$ acts on $\bZ^4 \cong \bZ^2 \oplus \bZ^2$ componentwise. It also acts on the equivalent lattice $\frac{1}{2} \bZ^4$, with each action giving the same integral representation $C_4 \times C_4 \rightarrow \GL_4(\bZ)$. Next we consider the intermediate lattice
\[
\bZ^4 \subsetneq L := \langle v_1 := (1, 0, 0, 0), \ v_2 := (0, 1, 0, 0), \ v_3 := (0, 0, 1, 0), \ v_4 := \frac{1}{2} (1, 1, 1, 1) \rangle  
\subsetneq \frac{1}{2} \bZ^4.
\]
Notice that $L$ contains $(0, 0, 0, 1) = 2 v_4 - (v_1 + v_2 + v_3)$ and therefore contains $( a/2, b/2, c/2, d/2)$ for $a, b, c, d \in 2 \bZ +1$. The action of $C_4 \times C_4$ on $\bZ^4$ restricts to $L$, giving a representation
\[
C_4 \times C_4 \rightarrow \GL(L) \cong \GL(\bZ^4)
\]
that is rationally (but not integrally) equivalent to to the original representation.  That it is not integrally equivalent follows because $L$ is not spanned by elements which are fixed by one of the two $C_4$ factors, whereas $\bZ^4$ is so spanned.  This is true even though there is a short exact sequence
\[
0 \to \bZ^2 \to L \to \bZ^2 \to 0, 
\]
with the first and third terms giving the standard rotation representations of the two factors of $C_4$. So there are nontrivial ways to extend integral representations.
\end{exam}

We complete the prove of Proposition \ref{prop:symplectic_tori} by ruling out the binary tetrahedral group. 

\begin{lem}\label{l:bintet-nores}
    The quotient $(\bC^\times)^4/ T$, where $T$ is the binary tetrahedral group acting symplectically on $(\bC^\times)^4$ does not admit a symplectic resolution of singularities. 
\end{lem}

\begin{proof}
First, observe that there may be many ways for $T$ to act symplectically on $(\bC^\times)^4$. We are fixing only the complex symplectic representation of $T$, so the embedding $T \subseteq \Sp_4(\bC)$. This is conjugate to a real representation, $T \subseteq \Sp_4(\bR)$, and this real representation is uniquely determined up to isomorphism.  One explicit way to understand it is to realize $T$ inside the quaternions $\bH$, as the group generated by $i,j,k,\frac{1}{2}(1+i+j+k)$.  Then, every integral representation, up to isomorphism, can be constructed by choosing a full (rank four) $T$-invariant lattice $\Lambda \subseteq \bH$.

We now consider the fixed points of $T$ on $(\bC^\times)^4$. First consider the element $i \in T$. Now $i$ acts on $\bQ^4$ without eigenvalues $\pm 1$, so this determines an action of the Gaussian integers, $\bZ[i] \cong \bZ[x]/(x^2+1)$, on $\Lambda$. This action is torsion free as $\Lambda$ is torsion free. Since $\bZ[i]$ is a PID, $\Lambda$ is a free $\bZ[i]$-module. So we can write the action in the standard way. The number of fixed points is four, the same as for block diagonal $90$-degree rotation matrices, which is the $\gamma_4$ case above.

Next, in order to admit a symplectic resolution, each such fixed point must also be fixed by a reflection. Since $i$ and any reflection generate $T$, we can assume that each of the above fixed points are fixed by all of $T$, at least to investigate the existence of a symplectic resolution.  

We next turn to elements of order six (i.e., $-1$ times a reflection).  For such an element $\gamma$ there is a saturated sublattice of rank 2, $\Lambda_0$, on which $\gamma$ acts by $-1$, and quotient lattice $\Lambda/\Lambda_0$ on which $\gamma$ acts as a $\bZ[\zeta]$-module, for $\zeta$ a primitive cube root of unity.  The latter is also a PID and hence $\Lambda/\Lambda_0$ is a free $\bZ[\zeta]$-module. So we can write the action of $\gamma$ in the standard way for an order six element. Overall the action of $\gamma$ is block-upper triangular with diagonal blocks $-I$ and the standard order-six action.  There are still 4 fixed points, since the order six action has no fixed points (the equations to be fixed give a unique solution, regardless of the choice of four fixed points corresponding to the sublattice, i.e., quotient torus).  By the preceding argument, each fixed point has stabilizer all of $T$.

However, the element $-I$ fixed sixteen points on the four-torus: all involutions.  So there are twelve points with stabilizer $C_2$; they form one $T$-orbit. Thus there is a nonresolvable $C_2$-singularity of $(\bC^\times)^4/T$.
\end{proof}

\begin{rem}
There is an obvious choice of lattice invariant under $T$ in the proof: the lattice of Hurwitz quaternions. For this choice one can see that the fixed points are exactly as in the proof.  Note that
there could exist a different lattice with different fixed points: this would happen only if the elements $i,j,k$ do not all share the same four fixed points. In this case, one of the non-identity fixed points of $i$ would also have to be fixed for $j$ and $k$, hence for all of $Q_8$, hence for all of $T$ since $Q_8$ is normal. So we would get two fixed points for $T$, and two fixed points each for the groups $\langle i\rangle, \langle j \rangle, \langle k \rangle \cong C_4$, which all form a single orbit of $T$. Aside from these eight points we would have eight more fixed points for $-I$, breaking up into subsets of two fixed points for minus each reflection and its inverse. These sets of size two have stabilizers $C_6$, also nonresolvable. They form two orbits of size four. So in this model we would get one nonresolvable $C_4$ singularity and two nonresolvable $C_6$ singularities (instead of one nonresolvable $C_2$ singularity). It would be interesting to see if an appropriate lattice can be found to construct such a quotient.
\end{rem}

\appendix
\section{The Weinstein splitting theorem in the singular case} 
\subsection{The splitting theorem}\label{s:appendix_Weinstein}
The celebrated Weinstein splitting theorem \cite[Theorem 2.1]{Weinstein83} gives a local decomposition of a smooth Poisson manifold into a symplectic part and a part with zero rank at the basepoint. This result goes through exactly the same in the case of analytic varieties, but for sake of completeness, we provide the statement and repeat the proof.  Note that a version of this for formal neighborhoods was proved in \cite{Kaledin06}, as a consequence of  Proposition 3.3 therein. By passing from analytic to formal neighborhoods, the result below also recovers  such a decomposition. 

\begin{thm}
Let $(X,\sigma)$ be a (real or complex) analytic  Poisson variety and $x \in X$. Then there is a neighborhood $U \ni x$ of $X$ and a pointed Poisson isomorphism $(U,x) \cong (S,s) \times (Z,z)$ where  $S$ is symplectic and the Poisson structure of $Z$ vanishes at $z$. 
\end{thm}

\begin{proof}
Let $\bK$ denote either $\bR$ or $\bC$ depending on whether we work in the real analytic or complex analytic setting. 
The proof is by induction on the rank of $\sigma$.  If this rank is zero, then $S$ must be a point, and the statement is obvious.  Assume the rank is nonzero. Restricting to a small enough analytic neighborhood $U$, we can find a Hamiltonian $f \in \mathcal{O}(U)$ such that $\xi_f|_x \neq 0$. Up to further restriction of $U$, we may assume it is a Zariski closed subset of a ball $B \subseteq \bK^n$ for some $n$ centered at $x=0$. Then there exists a vector field $\widetilde \xi_f$ on $B$  which is tangent to $U$ and restricts there to $\xi_f$.
Up to shrinking $B$ and $U$, we may pick coordinates $x_1, \ldots, x_n$ on $B$ centered at $0$ with $\widetilde{\xi_f}(x_1)=1$.  
Let $g := x_1|_{U}$. Then $\{f,g\} = 1$, and hence $[\xi_f, \xi_g] = \xi_1 = 0$.  Since $\xi_f(f)=0=\xi_g(g)$, whereas $\xi_f(g)=1=-\xi_g(f)$, $\xi_f$ and $\xi_g$ are linearly independent at $0$. 

Since $\widetilde{\xi_f}$ is nonvanishing on $B$, by the Frobenius integrability theorem (for the holomorphic case, see \cite[2.26]{Voisin-hodge}), we can locally integrate the distribution spanned by $\widetilde{\xi_f}$ to a submersion $B \to \bK^{n-1}$ with $\widetilde{\xi_f}$ tangent to the fibers, up to shrinking $B$ and $U$.
This gives local coordinates $x_2, \ldots, x_n$ annihilated by $\widetilde{\xi_f}$. As a result,  $\widetilde{\xi_f} = h \partial_1$ for some function $h$, which by assumption is nonvanishing in $B$, and hence invertible. 
Then the coordinates $\int h^{-1} \text{d} x_1, x_2, \ldots, x_n$ have the property that $\widetilde{\xi_f}=\partial_1$.

Now, the ideal $I_X$ of $X$ is locally generated at $x=0$ by some functions $g_1,\ldots,g_k$ independent of $x_1$.  Shrinking $U$ and $B$ if necessary, we can assume these generate the ideal $I_U$ of $U \subseteq B$, globally on $B$. Extend $\xi_g$ to a vector field $\widetilde{\xi_g}$ on $B$ tangent to $U$. Note that $[\partial_1,\xi]$ vanishes on $U$ if and only if $\xi$ is a sum of a vector field constant in $x_1$ (i.e., commuting with $\partial_1$) and a vector field itself vanishing on $U$ (as $I_U$ is generated by functions independent of $x_1$).  So up to changing the extension $\widetilde{\xi_2}$ of $\xi_2$, we can assume that actually $[\widetilde{\xi_f}, \widetilde{\xi_g}]=0$.  

Then by the Frobenius integrability theorem again, we can find coordinates $x_3, \ldots, x_n$ with $\widetilde{\xi_f}(x_i)=0=\widetilde{\xi_g}(x_i)$ for $3 \leq i \leq n$ (again up to shrinking $B$ and $U$).
Since $\widetilde{\xi_f}, \widetilde{\xi_g}$ are tangent to $U$ we get that $I_U$ is generated by functions in $x_3, \ldots, x_n$.  Moreover, we can further change the coordinates to $g, f, x_3, \ldots, x_n$, since $\widetilde{\xi_f}(g) = 1 = - \widetilde{\xi_g}(f)$, whereas $\widetilde{\xi_f}(f)=0=\widetilde{\xi_g}(g)$, so that the derivatives of $f,g$ are linearly independent at $x$ to those of $x_3, \ldots, x_n$.  

In this new coordinate system we get a Poisson isomorphism $(U,x) \cong (S,s) \times (U',x)$, with $(S,s)$ symplectic of dimension two and $(U',x)$ Poisson, since $\{f,g\}=1$ and $\{f,x_i\}=0=\{g,x_i\}, i \geq 3$.
 
Iterating the procedure, we will eventually obtain the theorem. \qedhere

\end{proof}
\subsection{A non-Poisson version of the splitting theorem}
Rewriting the essential part of the argument of the Weinstein splitting theorem in the non-Poisson setting we arrive at the following, which we would imagine is known:

\begin{lem}
    Let $X$ be a (real or complex) analytic variety and $\xi_1, \ldots, \xi_m$ vector fields on $X$ which commute. Suppose that at some $x \in X$, $\xi_1|_x, \ldots, \xi_m|_x \in T_x X$ are linearly independent. Then in some local analytic neighborhood $U$ of $x$, the $\xi_i$ integrate to give a decomposition $U \cong D \times Z$ for $D$ a polydisc of dimension $m$, with coordinates $x_1, \ldots, x_m$, such that $\xi_i = \partial_i$.
\end{lem}
\begin{proof}
    We can take an analytic neighborhood $U$ of $x$ which is isomorphic to a Zariski closed subset of an open ball  $B \subseteq \bK^n$ centered at $0$, for $\bK = \bR$ or $\bC$. We will think of $U$ as a subset of $B$ with $x \in U$ equal to $0 \in B$.  We can pick $U$ small enough that $\xi_1, \ldots, \xi_m$ are linearly independent on $U$. Let $\widetilde{\xi_1}, \ldots, \widetilde{\xi_m}$ be vector fields on $B$ which are tangent to $U$ and restrict there to $\xi_1, \ldots, \xi_m$. Applying the Frobenius integrability theorem (for the holomorphic case, again see \cite[2.26]{Voisin-hodge}) to $\widetilde{\xi_1}$, we get a local submersion at $x=0$, $B \to \bK^{n-1}$ (up to shrinking $B$ and $U$) with $\widetilde \xi_1$ tangent to the fibers. This defines
    coordinates $x_1, x_2, \ldots, x_n$ on $B$ centered at $0$ with $\widetilde \xi_1 = h \partial_1$ for some function $h$ nonvanishing on $B$.  Changing coordinates $x_1 \mapsto \int h^{-1} \text{d} x_1$ (if necessary after shrinking $U$), we get $\widetilde \xi_1 = \partial_1$, as desired.

    Next, since $\widetilde \xi_1$ is tangent to $X$, the ideal $I_U$ is generated near $x \in \bK^n$ by some functions in $x_2, \ldots, x_n$ (i.e., independent of $x_1$). We can assume these globally generate up to shrinking $U$ and $B$. Then because $[\widetilde \xi_1, \widetilde \xi_2]$ vanishes on $U$, we see that $\widetilde \xi_2$ must be a sum of a vector field commuting with $\widetilde \xi_1 = \partial_1$ and one vanishing on $U$.  So we can find another choice of $\widetilde \xi_2$ (still restricting to $\xi_2$ on $X$)  such that $[\widetilde {\xi_1}, \widetilde {\xi_2}] = 0$. Now applying the Frobenius integrability theorem again, we can assume that $\widetilde{\xi_1}, \widetilde{\xi_2}$ both annihilate $x_3, \ldots, x_n$. As before we can assume $\widetilde{\xi_1}=\partial_1$. Since they are linearly independent, $\widetilde{\xi_2} = f \partial_1 + g \partial_2$  with $f,g$ independent of $x_1$ and $g$ nonvanishing on $B$.  Up to change of coordinates $x_1 \mapsto x_1 - \int f g^{-1} \text{d} x_2$, we can assume $f=0$, and up to change of coordinates $x_2 \mapsto \int g^{-1} \text{d} x_2$, we can assume $g=1$.  

    Inductively, if we have coordinates such that $\widetilde{\xi_1} = \partial_1, \ldots, \widetilde{\xi_k} = \partial_k$, then since these are tangent to $U$, the ideal $I_U$ is locally generated at $x$ by functions of $x_{k+1}, \ldots, x_n$.  We can make another choice of $\widetilde{\xi_{k+1}}$ (restricting to $\xi_{k+1}$ on $X$) which commutes with $\widetilde{\xi_1}, \ldots, \widetilde{\xi_k}$.  Applying the Frobenius integrability theorem we have new coordinates $x_{k+2}, \ldots, x_n$ annihilated by $\widetilde{\xi_1}, \ldots, \widetilde{\xi_{k+1}}$. We can  arrange that the latter vector fields are $\partial_1, \ldots, \partial_k$,  together with something of the form $f_1 \partial_1 + \cdots + f_{k+1} \partial_{k+1}$, with $f_{k+1}(0) \neq 0$. Now changes of coordinates as before replace $x_1, \ldots, x_{k+1}$ by new coordinates so that $\widetilde{\xi_i} = \partial_i$ for $1 \leq i \leq k+1$, completing the inductive step.

    By induction we thus find coordinates on a ball about $x=0$ in $\bK^n$ so that $\widetilde{\xi_i} = \partial_i$ for all $1 \leq i \leq m$. The ideal $I_X$ is generated by functions in the $x_i$ for $i > m$.  We obtain an analytic neighborhood of $x \in X$  isomorphic to a product of an open polydisc $D \subseteq \bK^m$  and another variety $Z$ (the vanishing locus of the generators of $I_X$ in a polydisc with coordinates $x_{m+1}, \ldots, x_n$), with $\xi_i$ the coordinate vector fields on $D$.
    \end{proof}

\bibliographystyle{alpha}
\bibliography{Local-to-Global}

\newcommand{\etalchar}[1]{$^{#1}$}
\begin{thebibliography}{BCHM10}

\bibitem[BC20]{BC20}
Gwyn Bellamy and Alastair Craw.
\newblock Birational geometry of symplectic quotient singularities.
\newblock {\em Invent. Math.}, 222(2):399--468, 2020.

\bibitem[BCHM10]{BCHM10}
Caucher Birkar, Paolo Cascini, Christopher~D. Hacon, and James McKernan.
\newblock Existence of minimal models for varieties of log general type.
\newblock {\em J. Amer. Math. Soc.}, 23(2):405--468, 2010.

\bibitem[BCR{\etalchar{+}}21]{BCRSW21}
Gwyn Bellamy, Alastair Craw, Steven Rayan, Travis Schedler, and Hartmut Weiss.
\newblock All 81 crepant resolutions of a finite quotient singularity are
  hyperpolygon spaces, 2021.

\bibitem[BCS22]{BCS22}
Gwyn Bellamy, Alastair Craw, and Travis Schedler.
\newblock Birational geometry of quiver varieties and other {G}{I}{T}
  quotients, 2022.

\bibitem[BCS23]{BCS21}
Tristan Bozec, Damien Calaque, and Sarah Scherotzke.
\newblock Calabi-{Y}au structures for multiplicative preprojective algebras.
\newblock {\em J. Noncommut. Geom.}, 17(3):783--810, 2023.

\bibitem[BD21]{BD-RelCY2}
Christopher Brav and Tobias Dyckerhoff.
\newblock Relative {C}alabi-{Y}au structures {II}: shifted {L}agrangians in the
  moduli of objects.
\newblock {\em Selecta Math. (N.S.)}, 27(4):Paper No. 63, 45, 2021.

\bibitem[Bea00]{Beauville}
Arnaud Beauville.
\newblock Symplectic singularities.
\newblock {\em Invent. Math.}, 139(3):541--549, 2000.

\bibitem[Bel09]{Bellamy09}
Gwyn Bellamy.
\newblock On singular {C}alogero-{M}oser spaces.
\newblock {\em Bull. Lond. Math. Soc.}, 41(2):315--326, 2009.

\bibitem[Bel16]{Bellamy16}
Gwyn Bellamy.
\newblock Counting resolutions of symplectic quotient singularities.
\newblock {\em Compositio Mathematica}, 152(1):99–114, 2016.

\bibitem[BS16]{BS16}
Gwyn Bellamy and Travis Schedler.
\newblock On the (non)existence of symplectic resolutions of linear quotients.
\newblock {\em Math. Res. Lett.}, 23(6):1537--1564, 2016.

\bibitem[BS21]{BS21}
Gwyn Bellamy and Travis Schedler.
\newblock Symplectic resolutions of quiver varieties.
\newblock {\em Selecta Math. (N.S.)}, 27(3):Paper No. 36, 50, 2021.

\bibitem[BS23]{BS19}
Gwyn Bellamy and Travis Schedler.
\newblock Symplectic resolutions of character varieties.
\newblock {\em Geom. Topol.}, 27(1):51--86, 2023.

\bibitem[CBS06]{CBS06}
William Crawley-Boevey and Peter Shaw.
\newblock Multiplicative preprojective algebras, middle convolution and the
  {D}eligne-{S}impson problem.
\newblock {\em Adv. Math.}, 201(1):180--208, 2006.

\bibitem[CGGS21]{CGGS19}
Alastair Craw, S{\o}ren Gammelgaard, \'{A}d\'{a}m Gyenge, and Bal\'{a}zs
  Szendr\H{o}i.
\newblock Punctual {H}ilbert schemes for {K}leinian singularities as quiver
  varieties.
\newblock {\em Algebr. Geom.}, 8(6):680--704, 2021.

\bibitem[Che05]{Cheltsov-factoriality}
Ivan Cheltsov.
\newblock On factoriality of nodal threefolds.
\newblock {\em J. Algebraic Geom.}, 14(4):663--690, 2005.

\bibitem[CY24]{CY-Hilbert}
Alastair Craw and Ryo Yamagishi.
\newblock The {L}e {B}ruyn--{P}rocesi theorem and {H}ilbert schemes.
\newblock 2024.

\bibitem[Dav21]{Davison21}
Ben Davison.
\newblock Purity and 2-{C}alabi-{Y}au categories, 2021.

\bibitem[EG02]{EG02}
Pavel Etingof and Victor Ginzburg.
\newblock Symplectic reflection algebras, {C}alogero-{M}oser space, and
  deformed {H}arish-{C}handra homomorphism.
\newblock {\em Invent. Math.}, 147(2):243--348, 2002.

\bibitem[Fog68]{Fogarty68}
John Fogarty.
\newblock Algebraic families on an algebraic surface.
\newblock {\em Amer. J. Math.}, 90:511--521, 1968.

\bibitem[Fuj22]{Fujino22}
Osamu Fujino.
\newblock Minimal model program for projective morphisms between complex
  analytic spaces, 2022.
\newblock \url{https://arxiv.org/abs/2201.11315}.

\bibitem[Gan12]{Ganatra12}
Sheel Ganatra.
\newblock {\em Symplectic {C}ohomology and {D}uality for the {W}rapped {F}ukaya
  {C}ategory}.
\newblock ProQuest LLC, Ann Arbor, MI, 2012.
\newblock Thesis (Ph.D.)--Massachusetts Institute of Technology.

\bibitem[Gil64]{Gilmartin-loc-contr}
M.~C. Gilmartin.
\newblock {\em Every analytic variety is locally contractible}.
\newblock ProQuest LLC, Ann Arbor, MI, 1964.
\newblock Thesis (Ph.D.)--Princeton University.

\bibitem[GK04]{GK04}
Victor Ginzburg and Dmitry Kaledin.
\newblock Poisson deformations of symplectic quotient singularities.
\newblock {\em Adv. Math.}, 186(1):1--57, 2004.

\bibitem[Gor03]{Gordon-baby}
Iain Gordon.
\newblock Baby {V}erma modules for rational {C}herednik algebras.
\newblock {\em Bull. London Math. Soc.}, 35(3):321--336, 2003.

\bibitem[HK97]{HausmannKnutson}
J.~Hausmann and A.~Knutson.
\newblock Polygon spaces and {G}rassmannians.
\newblock {\em Enseign. Math. (2)}, 43(1-2):173--198, 1997.

\bibitem[Jos97]{Joseph}
Anthony Joseph.
\newblock On a {H}arish-{C}handra homomorphism.
\newblock {\em Comptes Rendus de l'Académie des Sciences - Series I -
  Mathematics}, 324(7):759--764, 1997.
\newblock Serie I: Mathematique.

\bibitem[Kal03]{Kaledin03}
D.~Kaledin.
\newblock On crepant resolutions of symplectic quotient singularities.
\newblock {\em Selecta Math. (N.S.)}, 9(4):529--555, 2003.

\bibitem[Kal06]{Kaledin06}
D.~Kaledin.
\newblock Symplectic singularities from the {P}oisson point of view.
\newblock {\em J. Reine Angew. Math.}, 600:135--156, 2006.

\bibitem[Kal09]{KaledinSurvey}
D.~Kaledin.
\newblock Geometry and topology of symplectic resolutions.
\newblock In {\em Algebraic geometry---{S}eattle 2005. {P}art 2}, volume~80 of
  {\em Proc. Sympos. Pure Math.}, pages 595--628. Amer. Math. Soc., Providence,
  RI, 2009.

\bibitem[Kel08]{Keller08}
Bernhard Keller.
\newblock Calabi-{Y}au triangulated categories.
\newblock In {\em Trends in representation theory of algebras and related
  topics}, EMS Ser. Congr. Rep., pages 467--489. Eur. Math. Soc., Z\"{u}rich,
  2008.

\bibitem[Kin94]{King94}
A.~D. King.
\newblock Moduli of representations of finite-dimensional algebras.
\newblock {\em Quart. J. Math. Oxford Ser. (2)}, 45(180):515--530, 1994.

\bibitem[KM98]{KollarMori}
J.~Koll\'ar and S.~Mori.
\newblock {\em Birational geometry of algebraic varieties}, volume 134 of {\em
  Cambridge Tracts in Mathematics}.
\newblock Cambridge University Press, Cambridge, 1998.
\newblock With the collaboration of C. H. Clemens and A. Corti, Translated from
  the 1998 Japanese original.

\bibitem[Kov17]{Kovacs-rational}
S\'{a}ndor~J. Kov\'{a}cs.
\newblock Rational singularities, 2017.

\bibitem[KP82]{KP-norm-cl}
Hanspeter Kraft and Claudio Procesi.
\newblock On the geometry of conjugacy classes in classical groups.
\newblock {\em Comment. Math. Helv.}, 57(4):539--602, 1982.

\bibitem[KS23]{KS20}
Daniel Kaplan and Travis Schedler.
\newblock Multiplicative preprojective algebras are 2-{C}alabi-{Y}au.
\newblock {\em Algebra Number Theory}, 17(4):831--883, 2023.

\bibitem[LNY23]{LNY23}
Penghui Li, David Nadler, and Zhiwei Yun.
\newblock Functions on the commuting stack via {L}anglands duality, 2023.

\bibitem[Los17]{Losev17}
Ivan Losev.
\newblock Bernstein inequality and holonomic modules.
\newblock {\em Adv. Math.}, 308:941--963, 2017.
\newblock With an appendix by Losev and Pavel Etingof.

\bibitem[Mil68]{Milnor-singular-hypersurfaces}
J.~Milnor.
\newblock {\em Singular points of complex hypersurfaces}.
\newblock Annals of Mathematics Studies, No. 61. Princeton University Press,
  Princeton, N.J., 1968.

\bibitem[MN19]{McN-Springer}
Kevin McGerty and Thomas Nevins.
\newblock Springer theory for symplectic {G}alois groups, 2019.

\bibitem[MO15]{MT15}
David Mumford and Tadao Oda.
\newblock {\em Algebraic geometry. {II}}, volume~73 of {\em Texts and Readings
  in Mathematics}.
\newblock Hindustan Book Agency, New Delhi, 2015.

\bibitem[Mum81]{Mumford81}
David Mumford.
\newblock {\em Algebraic geometry. {I}}, volume 221 of {\em Grundlehren der
  Mathematischen Wissenschaften [Fundamental Principles of Mathematical
  Sciences]}.
\newblock Springer-Verlag, Berlin-New York, 1981.
\newblock Complex projective varieties, Corrected reprint.

\bibitem[Nam01]{NamikawaNote}
Y~Namikawa.
\newblock A note on symplectic singularities.
\newblock {\em arXiv}, 0101028, 2001.

\bibitem[Nam10]{Namikawa10}
Yoshinori Namikawa.
\newblock Poisson deformations of affine symplectic varieties, {II}.
\newblock {\em Kyoto J. Math.}, 50(4):727--752, 2010.

\bibitem[Nam11]{Namikawa11}
Yoshinori Namikawa.
\newblock Poisson deformations of affine symplectic varieties.
\newblock {\em Duke Mathematical Journal}, 156(1), jan 2011.

\bibitem[Nam22]{Namikawa-nilpotentcover}
Yoshinori Namikawa.
\newblock Birational geometry for the covering of a nilpotent orbit closure.
\newblock {\em Selecta Math. (N.S.)}, 28(4):Paper No. 75, 59, 2022.

\bibitem[Rei80]{Reid-c3f}
Miles Reid.
\newblock Canonical {$3$}-folds.
\newblock In {\em Journ\'{e}es de {G}\'{e}ometrie {A}lg\'{e}brique d'{A}ngers,
  {J}uillet 1979/{A}lgebraic {G}eometry, {A}ngers, 1979}, pages 273--310.
  Sijthoff \& Noordhoff, Alphen aan den Rijn---Germantown, Md., 1980.

\bibitem[Sha05]{Shaw05}
Peter Shaw.
\newblock Generalisations of preprojective algebras.
\newblock Ph.D. thesis, Univeristy of Leeds- available at
  \url{https://www.math.uni-bielefeld.de/~wcrawley/shaw-thesis.pdf}, 2005.

\bibitem[ST19]{ST19}
Travis Schedler and Andrea Tirelli.
\newblock Symplectic resolutions for multiplicative quiver varieties and
  character varieties for punctured surfaces, 2019.

\bibitem[Tre09]{TrEP}
David Treumann.
\newblock Exit paths and constructible stacks.
\newblock {\em Compos. Math.}, 145(6):1504--1532, 2009.

\bibitem[Ver00]{Verbitsky00}
Misha Verbitsky.
\newblock Holomorphic symplectic geometry and orbifold singularities.
\newblock {\em Asian J. Math.}, 4(3):553--563, 2000.

\bibitem[Voi02]{Voisin-hodge}
Claire Voisin.
\newblock {\em Hodge theory and complex algebraic geometry. {I}}, volume~76 of
  {\em Cambridge Studies in Advanced Mathematics}.
\newblock Cambridge University Press, Cambridge, 2002.
\newblock Translated from the French original by Leila Schneps.

\bibitem[Wei83]{Weinstein83}
Alan Weinstein.
\newblock The local structure of {P}oisson manifolds.
\newblock {\em J. Differential Geom.}, 18(3):523--557, 1983.

\bibitem[Zar43]{Zariski43}
Oscar Zariski.
\newblock Foundations of a general theory of birational correspondences.
\newblock {\em Transactions of the American Mathematical Society},
  53(3):490--542, 1943.

\end{thebibliography}

\end{document}